\documentclass[11pt,a4paper,Color]{amsart}
\usepackage{bbm}
\usepackage{stmaryrd}
\usepackage{cases}
\usepackage{mathrsfs}
\usepackage{amsfonts}
\usepackage{amsmath}
\usepackage{empheq}
\usepackage{color}
\usepackage{amssymb}
\newtheorem{theorem}{Theorem}[section]
\newtheorem{lemma}[theorem]{Lemma}

\newtheorem{corollary}[theorem]{Corollary}
\newtheorem{proposition}[theorem]{Proposition}
\theoremstyle{definition}
\newtheorem{definition}[theorem]{Definition}
\theoremstyle{remark}
\newtheorem{remark}[theorem]{Remark}
\newcommand{\dif}{\mathrm{d}}
\newcommand{\Hess}{\mathrm{Hess}}
\newcommand{\Id}{\mathrm{Id}}
\newcommand{\R}{\mathrm{R}}
\newcommand{\maxd}{\mathrm{maxd}}
\newcommand{\ta}{\mathrm{ta}}
\newcommand{\dist}{\mathrm{dist}}

\newcommand{\s}{\mathrm{s}}
\newcommand{\cc}{\mathrm{c}}
\newcommand{\co}{\mathrm{co}}

\newcommand{\e}{\mathrm{e}}

\newcommand{\graph}{\mathrm{graph}}
\newcommand{\tr}{\mathrm{tr}}

\def\ds{\displaystyle}
\def\vle{{\rm vol}}

 \numberwithin{equation}{section}

\begin{document}

\title[V.P.$m$th.P.M.C.F in hyperbolic spaces]
{Volume-Preserving flow by powers of the $m$th mean curvature in the
hyperbolic space }

\author{Shunzi Guo}
\address{School of Mathematics and Statistics, Minnan Normal University, Zhangzhou, 363000, People's Republic of China;
School of Mathematics and Computer Science, Hubei University, Wuhan,
430062, People's Republic of China} \email{guoshunzi@yeah.net}
\thanks{This work was partially supported by National Natural Science Foundation of
China (NSFC) (Grants 11171096 and 10917055)}

\author{Guanghan Li}
\address{School of Mathematics and Computer Science, Hubei University, Wuhan 430062, People's Republic of China}
\email{liguanghan@163.com}

\author{Chuanxi Wu}
\address{School of Mathematics and Computer Science, Key Laboratory of Applied Mathematics of
Hubei Province, Hubei University, Wuhan 430062, People's Republic of
China} \email{cxwu@hubu.edu.cn}

\renewcommand{\subjclassname}{%
  \textup{2010} Mathematics Subject Classification}



\begin{abstract}
This paper concerns closed hypersurfaces of dimension $n(\geq 2)$ in
the hyperbolic space ${\mathbb{H}}_{\kappa}^{n+1}$ of constant
sectional curvature $\kappa$ evolving in direction of its normal
vector, where the speed is given by a power $\beta (\geq 1/m)$ of
the $m$th mean curvature plus a volume preserving term, including
the case of powers of the mean curvature and of the $\mbox{Gau\ss}$
curvature. The main result is that if the initial hypersurface
satisfies that the ratio of the biggest and smallest principal
curvature is close enough to $1$ everywhere, depending only on $n$,
$m$, $\beta$ and $\kappa$, then under the flow this is maintained,
there exists a unique, smooth solution of the flow for all times,
and the evolving hypersurfaces exponentially converge to a geodesic
sphere of ${\mathbb{H}}_{\kappa}^{n+1}$, enclosing the same volume
as the initial hypersurface.
\end{abstract}
\maketitle
\tableofcontents

\section{Introduction}
Let $M^n$ be a smooth, compact oriented manifold of dimension $n
(\geq 2)$ without boundary, $(N^{n+1},\bar{g})$ be an
$(n+1)$-dimensional completed Riemannian manifold, and
$X_{0}:M^{n}\rightarrow {N}^{n+1}$ a smooth immersion. Consider a
one-parameter family of smooth immersions: $X_{t}:M^{n}\rightarrow
{N}^{n+1}$ evolving according to
\begin{equation}\label{vpmthpcf}
\left\{
\begin{array}{ll}
\frac{\partial}{\partial t}\mathrm{X}\left(p,t\right)
    = \{\bar{F}(t)-F(\lambda(\mathscr{W}\left(p,t\right)))\} \nu\left(p,t\right),& p\in M^{n},\\[2ex]
X(\cdot,0)=X_{0}(\cdot),
\end{array}\right.
\end{equation}
where $\nu\left(p,t\right)$ is the outer unit normal to
$M_{t}=X_{t}(M^{n})$ at the point $\mathrm{X}\left(p,t\right)$ in
the tangent space $TN^{n+1}$,
$\mathscr{W}_{-\nu}\left(p,t\right)=-\mathscr{W}_{\nu}\left(p,t\right)$
is the matrix of Weingarten map on the tangent space $TM^{n}$
induced by $X_{t}$, $\lambda$ is the map from $T^{*}M^{n}\otimes
TM^{n}$ to $\mathbb{R}^{n}$ which gives the eigenvalues of the map
$\mathscr{W}$, $F$ is a smooth symmetric function, and $\bar{F}(t)$
is the average of $F$ on $M_{t}$:
\begin{equation}\label{ave barF}
\bar{F}(t)=\frac{\int_{M_{t}}F(\lambda(\mathscr{W}))\dif\mu_{t}}{\int_{M_{t}}\dif\mu_{t}},
\end{equation}
where $\dif\mu_{t}$ denotes the surface area element of $M_{t}$. As
is clear for the presence of the global term $\bar{F}(t)$ in
equation \eqref{vpmthpcf}, the flow keeps the volume of the domain
$\Omega_{t}$ enclosed by $M_{t}$ constant.

This paper considers the flow \eqref{vpmthpcf} with the speed
$F(\lambda)$ given by a power of an $m$th mean curvature, namely
\begin{equation}\label{def barF}
F(\lambda_{1},\dots,\lambda_{n})=H_{m}^{\beta},
\end{equation}
where $(\lambda_{1},\dots,\lambda_{n})$ are the principal curvatures
of the evolving hypersurfaces $M_{t}$, and for any $m=1,\dots, n$,
the $m$th mean curvature $H_{m}$ is the average of the $m$th
elementary symmetric functions $E_{m}$, namely
\begin{equation}\label{def Hm}
H_{m}={ n\choose m }^{-1}E_{m}=\frac{m!(n-m)!}{n!}\sum\limits_{1
\leq i_{1}< \cdots < i_{m} \leq n}\lambda_{i_{1}}\cdots
\lambda_{i_{m}}.
\end{equation}
Obviously $H_{1}=H/n$ and $H_{n}=K$, where $H$ and $K$ denote the
mean curvature and the Gau\ss-Kronecker curvature respectively.

For the flow \eqref{vpmthpcf} without the volume constraint term
$\bar{F}(t)$, in the case when $N^{n+1}$ is the Euclidean space
${\mathbb{R}}^{n+1}$, there are many papers which consider the
evolution of convex hypersurfaces, of particular interest here is in
the analysis of the flow \eqref{vpmthpcf} where the speed
$F(\lambda)$ is homogeneous of degree one in the principal
curvatures, beginning with a classical result of Huisken
\cite{Hui84} who proved that any closed convex hypersurface under
mean curvature flow shrinks to a round point in finite time
(Huisken's theorem may be considered as an extension of the theorem
of Gage and Hamilton \cite{Ga-Ha} to dimensions bigger than one),
and including the similar results on the $n$th root of
Gau\ss-Kronecker curvature \cite{Cho85}, the square root of scalar
curvature \cite{Cho87}, and a large family of other speeds
\cite{And94,And07}. The first such result with degree of homogeneity
greater than one was due to Chow \cite{Cho85} who considered flow by
powers of the Gau\ss-Kronecker curvature. He proved that the
evolving hypersurfaces by $K^{\beta}$ with $\beta\geq 1/n$ become
spherical as they shrink to a point provided the initial
hypersurface $M_{0}$ is sufficiently pinched, we also mention that
Tso \cite{Tso} showed the same result for the Gau\ss-Kronecker
curvature flow and Andrews \cite{And00} proved that the limit of the
solutions under $K^{\beta}$-flow with $\beta \in (\frac{1}{n+2},
\frac{1}{n}]$ evolve purely by homothetic contraction to a point in
finite time. Later such results were proved by Schulze \cite{Sch06}
for the flow by powers of the mean curvature, by Alessandroni and
Sinestrari \cite{AC10} for the flow by powers of the scalar
curvature, by Andrews \cite{And12} for the flow of convex
hypersurfaces with pinched principal curvatures by high powers of
curvature, and for such flows in the special case of surfaces in
three-dimensional spaces \cite{And99, And10, Schn05, Sch06}, where
the lower dimension allows a more complete understanding of the
equation for the evolution of the second fundamental form. However,
when $N^{n+1}$ is a more general Riemannian manifold, there are few
results on the behavior of these flows: Huisken \cite{Hui86}
extended the result of \cite{Hui84} in ${\mathbb{R}}^{n+1}$ to
compact hypersurfaces in general Riemannian manifolds with suitable
bounds on curvature. Andrews \cite{Andr94} has considered a flow
which takes any compact hypersurface with principal curvatures
greater than $\sqrt{c}$ with $c>0$ in a Riemannian background space
with sectional curvatures at least $-c$, and converges to a round
point in finite time.

The volume-preserving versions of these flows are the flows
\eqref{vpmthpcf}-\eqref{ave barF} with an extra term $\bar{F}(t)$
which balances the contraction. In the case of volume-preserving
mean curvature flow, Huisken \cite{Hui87} showed that convex
hypersurfaces remain convex for all time and converge exponentially
fast to round spheres (the corresponding result for curves in the
plane is due to Gage \cite{Ga86}), while Andrews \cite{And01}
extended this result to the smooth anisotropic mean curvature flow,
and McCoy showed similar results for the surface area preserving
mean curvature flow \cite{McC03} and the mixed volume preserving
mean curvature flows \cite{McC04}. The volume-preserving flow has
been used to study constant mean curvature surfaces between parallel
planes \cite{At97,At03} and to find canonical foliations near
infinity in asymptotically flat spaces arising in general relativity
\cite{HY96} (Rigger \cite{Rig04} showed analogous results in the
asymptotically hyperbolic setting). If the initial hypersurface is
sufficiently close to a fixed Euclidean sphere (possibly
non-convex), Escher and Simonett \cite{ES98} proved that the flow
converges exponentially fast to a round sphere, a similar result for
average mean convex hypersurfaces with initially small traceless
second fundamental form is due to Li \cite{Li09}. For a general
ambient manifold, Alikakos and Freire \cite{AF03} proved long time
existence and convergence to a constant mean curvature surface under
the hypotheses that the initial hypersurface is close to a small
geodesic sphere and that it satisfies some non-degenerate
conditions. While, Cabezas-Rivas and Miquel exported the Euclidean
results of \cite{At97,At03} to revolution hypersurfaces in a
rotationally symmetric space \cite{CR-M09}, and showed the same
results as Huisken \cite{Hui87} for a hyperbolic background space
\cite{CR-M07} by assuming the initial hypersurface is
horospherically convex (the definition will be given later).

On the other hand, there are few results on speeds different from
the mean curvature: McCoy \cite{McC05} proved the convergence to a
sphere for a large class of function $F$ homogeneous of degree one
(including the case $F=H_{m}^{\beta}$ with $m\beta=1$), Makowski
showed that the mixed volume preserving curvature flow for a
function $F$ homogeneous of degree one, starting with a compact and
strictly horospheres-convex hypersurface in the hyperbolic space
exponentially converges to a geodesic sphere \cite{Mak12}, and the
volume preserving curvature flow in Lorentzian manifolds for $F$ as
a function with homogeneous of degree one exponential converges to a
hypersurface of constant $F$-curvature \cite{Mak13} (moreover,
stability properties and foliations of such a hypersurface  was also
examined). In 2010 Cabezas-Rivas and Sinestrari \cite{CS10} studied
the deformation of convex hypersurfaces in ${\mathbb{R}}^{n+1}$ by a
speed of the form \eqref{def Hm} for some power $\beta \geq 1/m$. In
this way $F$ is a homogeneous function of the curvatures with a
degree $m\beta \geq 1$. In particular, they proved the following
\begin{theorem}\label{theo1.1}
For $m\in \{1,\dots,n\}$, $m\beta \geq 1$ there exists a positive
constant $\mathcal{C}=\mathcal{C}(n, m,\beta ) < 1/n^n $ such that
the following holds: If the initial hypersurface of
${\mathbb{R}}^{n+1}$ is pinched in the sense that
\begin{equation}\label{initial pinched}
K(p)> \mathcal{C}H^n(p)>0\quad \text{for all}\quad p \in M^n,
\end{equation}
then the flow \eqref{vpmthpcf}-\eqref{def barF} with $F$ given by
\eqref{def Hm}, has a unique and smooth solution for all times,
inequality \eqref{initial pinched} remains everywhere on the
evolving hypersurfaces $M_{t}$ for all $t>0$ and the $M_{t}$'s
converge, exponentially in the $C^{\infty}$-topology, to a round
sphere of the same volume as $M_{0}$.
\end{theorem}

However, the results of \cite{CS10} do not close relate to the
ambient space, we face the challenges of extending the above results
to hypersurface to more general ambient spaces. But not every
Riemnnian manifold is well suited to deal with the situation
analogous to the setting in Euclidean spaces. We want to consider
the case that the ambient space is a simply connected Riemannian
manifold of constant sectional curvature $\kappa (<0)$ whose flow
behaves quite differently compared to the Euclidean space to a
certain extent.

Set $a=\sqrt{|\kappa|}$. $N_{\kappa}^{n+1}$ is isometric to the
hyperbolic space ${\mathbb{H}}_{\kappa}^{n+1}$ of radius $1/a$:
$${\mathbb{H}}_{\kappa}^{n+1}:= \{p \in L^{n+2}: \left\langle p, p\right\rangle = -\frac{1}{a^2}\}.$$
Here $ (L^{n+2},\langle \cdot, \cdot\rangle)$ denotes the
$(n+2)$-dimensional Lorentz-Minkowski space. To consider the flow
\eqref{vpmthpcf}-\eqref{def barF} in $N_{\kappa}^{n+1}$ is then
equivalent to consider the flow \eqref{vpmthpcf}-\eqref{def barF} in
${\mathbb{H}}_{\kappa}^{n+1}$. Now, it is necessary to provided some
definitions as in \cite{Bor,CR-M07} as following.
\begin{definition}
A horosphere $\mathcal {H}$ of ${\mathbb{H}}_{\kappa}^{n+1}$ is the
limit of a geodesic sphere of ${\mathbb{H}}_{\kappa}^{n+1}$ as its
center goes to the infinity along a fixed geodesic ray.
\end{definition}
\begin{definition}
An horoball $\mathscr{H}$ is the convex domain whose boundary is a
horosphere.
\end{definition}

\begin{definition}
A hypersurface $M$ of ${\mathbb{H}}_{\kappa}^{n+1}$ is said to be
convex by horospheres ($h$-convex for short) if it bounds a domain
$\Omega$ satisfying that for every $p\in M = \partial \Omega$, there
is a horosphere $\mathcal {H}$ of ${\mathbb{H}}_{\kappa}^{n+1}$
through $p$ such that $\Omega$ is contained in $\mathscr{H}$ of
${\mathbb{H}}_{\kappa}^{n+1}$ bounded by $\mathcal {H}$.
\end{definition}
\begin{remark}
In fact, Currier in \cite{Cu} showed that $h$-convex immersions of
smooth compact hypersurfaces are embedded spheres, and Borisenko and
Miquel in \cite{Bor} showed that horosphere $\mathcal {H}$ of
${\mathbb{H}}_{\kappa}^{n+1}$ is weakly (strictly) $h$-convex if and
only if all its principal curvatures are (strictly) bounded from
below by $a$ at each point.
\end{remark}
Most of the literature mentioned above requires a pinching condition
on the initial hypersurface, so that parabolical maximum principles,
an important tool in the investigation of evolution equations, can
be used to deduce that they can converge and become spherical in
shape as the final time is approached under these flows. It is
well-known that in hyperbolic spaces the negative curvature of the
background space produces terms which introduces the maximum
principles either fail or become more subtle for our flow
\eqref{vpmthpcf}-\eqref{def barF}. So for our purposes a challenge
in the hyperbolic ambient setting is how to find a suitable pinching
condition on the initial hypersurfaces. However in the hyperbolic
space there is an intuitive example, as pointed out by Cabezas-Rivas
and Miquel in \cite{CR-M07}: a geodesic sphere, moving outward in
the radial direction with the speed $H^{\beta}_{m}$, its normal
curvature decreases, and it becomes nearer and nearer to that of
horosphere, but it never gets $h$-convex. This fact leads us to hope
for the result by choosing a suitable convex hypersurface which is
sufficiently positively curved to overcome the obstructions from the
negative curvature imposed by the ambient spaces like that of space
of Cabezas-Rivas and Sinestrari \cite{CS10}. More precisely, denote
the turbulent second fundamental form by $ \tilde{h}_{ij}:=h_{ij}-a
g_{ij}$, then the turbulent mean curvature $\tilde{H}=H-na$ and the
turbulent $\mbox{Gau\ss}$ curvature
$\tilde{K}=\det\{\tilde{h}_{i}^{j}\}$. In this paper the turbulent
geometric quantities are distinguished by a tilde. Compared with the
pinching condition $K(p)> \mathcal{C}H^n(p)>0$ on the initial
hypersurface in Theorem \ref{theo1.1}, which is analogous to the
initial pinching condition in Chow \cite{Cho85} and Schulze
\cite{Sch06}, it is natural to impose a pinching condition:$
\tilde{K}> C^{*}\tilde{H}^n >0$ on the initial hypersurfaces of
${\mathbb{H}}_{\kappa}^{n+1}$, where $ C^{*}$ is a suitable positive
constant. It is shown in Section\,4 that the condition $ \tilde{K}>
C^{*}\tilde{H}^n >0$ on a closed hypersurface implies in particular
the $h$-convexity of the hypersurface. The aim of this paper is to
achieve such extension of the above Theorem \ref{theo1.1} of
Cabezas-Rivas and Sinestrari \cite{CS10} in Hyperbolic case.
Precisely, we prove the following
\begin{theorem}[main theorem]\label{main theorem}
For $m\in \{1,\dots,n\}$, $m\beta \geq 1$ there exists a positive
constant $C^{*}=C^{*}(a,n, m,\beta ) < 1/n^n $ such that the
following holds: If the initial hypersurface of
${\mathbb{H}}_{\kappa}^{n+1}$ is pinched in the sense that
\begin{equation}\label{ini pinching}
\tilde{K}(p)> C^{*}\tilde{H}^n(p)>0\quad \text{for all}\quad p \in
M^n,
\end{equation}
then the flow \eqref{vpmthpcf}-\eqref{def barF} with $F$ given by
\eqref{def Hm}, has a unique and smooth solution for all times,
inequality \eqref{ini pinching} remains everywhere on the evolving
hypersurfaces $M_{t}$ for all $t>0$ and the $M_{t}$'s converge,
exponentially in the $C^{\infty}$-topology, to a geodesic sphere of
${\mathbb{H}}_{\kappa}^{n+1}$ enclosing the same volume as $M_{0}$.
\end{theorem}

Our analysis follows the framework of \cite{CS10}, we make
modifications to consider our problem for the background space. The
rest of the paper is organized as follows: Section\,2 first gives
some useful preliminary results employed in the remainder of the
paper. Section\,3 contains details of the short time existence of
the flow \eqref{vpmthpcf}-\eqref{def barF} and the induced evolution
equations of some important geometric quantities and the
corresponding turbulent quantities, this requires only minor
modifications of Euclidean case due to the background curvature. In
Section\,4  applying the maximum principle to the evolution equation
of the turbulent quantity $\tilde{K}/\tilde{H}^n$ gives that if the
initial hypersurface is pinched good enough then this is preserved
for $t>0$ as long as the flow \eqref{vpmthpcf}-\eqref{def barF}
exists. This is a fundamental step in our procedure as in most of
the literature quoted above. Furthermore, Section\,5 proves the
uniform boundedness of the speed $F$ by following a method which was
firstly used by Tso \cite{Tso}. Using more sophisticated results for
fuly nonlinear elliptic and parabolic partial differential
equations, Section\,6 obtains uniform bounds on all derivatives of
the curvature and proves long time existence of the flow
\eqref{vpmthpcf}-\eqref{def barF}. Finally Section\,7, following the
idea in \cite{Mak12}, obtains the lower bound for $\tilde{H}$, which
we infer from a Harnack inequality due to Makowski \cite{Mak12}, the
estimates obtained so far will then allow us to prove that these
elvoving hypersurfaces converge to a geodesic sphere of
${\mathbb{H}}_{\kappa}^{n+1}$ smoothly and exponentially.

\section{Notation and preliminary results}

From now on, use the same notation as in \cite{CR-M09,Hui84,Sch05}
for local coordinates $\{x^{i}\}$, $1\leq i \leq n$, near $p\in
M^{n}$ and $\{y^{\alpha}\}$, $0 \leq \alpha, \beta \leq n$, near
$F(p)\in {\mathbb{H}}_{\kappa}^{n+1}$. Denote by a bar all
quantities on ${\mathbb{H}}_{\kappa}^{n+1}$, for example by
$\bar{g}=\{\bar{g}_{\alpha \beta}\}$ the metric, by
$\bar{g}^{-1}=\{\bar{g}^{\alpha \beta}\}$ the inverse of the metric,
by $\bar{\nabla}$ the covariant derivative, by $\bar{\Delta}$ the
rough Laplacian, and by
$\bar{\R}=\{\bar{\R}_{\alpha\beta\gamma\delta}\}$ the Riemann
curvature tensor. Components are sometimes taken with respect to the
tangent vector fields $\partial_{\alpha}(=\frac{\partial}{\partial
y^{\alpha}})$ associated with a local coordinate $\{y^{\alpha}\}$
and sometimes with respect to a moving orthonormal frame
$e_{\alpha}$, where $\bar{g}(e_{\alpha},
e_{\beta})=\delta_{\alpha\beta}$. The corresponding geometric
quantities on $M^{n}$ will be denoted by $g$ the induced metric, by
$g^{-1}$ the inverse of $g$, $\nabla, \Delta, \R,
\partial_{i}$ and $e_{i}$ the covariant derivative, the rough Laplacian, the curvature tensor,
the natural frame fields and a moving orthonormal frame field,
respectively. Then further important quantities are the second
fundamental form $A(p)=\{h_{ij}\}$ and the Weingarten map
$\mathscr{W}=\{g^{ik}h_{kj}\}=\{h^{i}_{j}\}$ as a symmetric operator
and a self-adjoint operator respectively. The eigenvalues
$\lambda_{1}(p)\leq \cdots \leq \lambda_{n}(p)$ of $\mathscr{W}$ are
called the principal curvatures of $X(M^{n})$ at $X(p)$. The mean
curvature is given by
\[H:=\tr_{g}{\mathscr{W}}=h^{i}_{i}=\sum_{i=1}^{n}\lambda_{i},\]
the square of the norm of the second fundamental form by
$$\bigl|A\bigr|^{2}:=\tr_{g}({\mathscr{W}^{t}\mathscr{W}})=h^{i}_{j}h^{j}_{i}=h^{ij}h_{ij}=\sum_{i=1}^{n}\lambda^{2}_{i},$$
and the Gau\ss-Kronecker curvature by
$$K:=\det(\mathscr{W})=\det\{h^{i}_{j}\}=\frac{\det\{h_{ij}\}}{\det\{g_{ij}\}}=\prod_{i=1}^{n}\lambda_{i}.$$
More generally, the $m$th elementary symmetric functions $E_{m}$ are
given by
$$E_{m}(\lambda)=\sum\limits_{1 \leq i_{1}< \cdots < i_{m} \leq n}\lambda_{i_{1}}\cdots
\lambda_{i_{m}} =\frac{1}{m!}\sum\limits_{
i_{1},\ldots,i_{m}}\lambda_{i_{1}}\cdots \lambda_{i_{m}},\ \ {\hbox
{for}}\  \lambda = (\lambda_{1},\ldots,\lambda_{n})\in
\mathbb{R}^{n},$$ and the $m$th mean curvature $H_{m}$ are given by
\eqref{def Hm}. Since $H_{m}$ is homogeneous of degree $m$, the
speed $F$ is homogeneous of degree $m\beta$ in the curvatures
$\lambda_{i}$. Denote the vector $(\lambda_{1},\dots, \lambda_{n})$
of $\mathbb{R}^{n}$ by $\lambda$ and the positive cone by
$\Gamma_{+}\subset \mathbb{R}^{n}$, i.e.
\[
\Gamma_{+}=\{\lambda=(\lambda_{1},\dots, \lambda_{n}):\lambda_{i}>0,
\forall\ i\}.
\]
It is clear that $H$, $K$, $H_{m}$, $F$ may be viewed as functions
of $\lambda$, or as functions of $A$, or as functions of
$\mathscr{W}$, or also functions of space and time on $M_{t}$. Since
the differentiability properties of these functions are the same in
our setting, we do not distinguish between these notions and write
always these functions the same letters in all cases. We use the
notation
\[
\dot{F}^{i}:=\frac{\partial F}{\partial \lambda_{i}},\ \
\dot{F}^{ij}:=\frac{\partial F}{\partial h_{ij}},\ \ \text{and}\ \
\dot{F}_{i}^{j}:=\frac{\partial F}{\partial h^{i}_{j}}.
\]
If $B$, $C$ are matrices, we write
\[
\dot{F}B = \dot{F}(B):=\dot{F}_{i}^{j}B_{j}^{i}\ \ \text{and} \ \
\ddot{F}(B,C):=\frac{\partial^{2} F}{\partial h^{j}_{i}\partial
h^{l}_{k}}B_{i}^{j}C_{k}^{l}.
\]
$\Hess_{\nabla}$ will denote the second tensorial derivative as a
$2$-covariant tensor. If it is contracted by the standard metric $g$
we have the standard Laplace-Beltrami operator: for any tensor $T$,
we write
\[
g^{-1}\Hess_{\nabla}(T) = g^{ij}\nabla_{i}\nabla_{j}T = \Delta T.
\]
More generally, given a $(2,0)$-tensor $w$, we denote
\[
\Delta_{w} T:= w^{ij}\nabla_{i}\nabla_{j}T.
\]
We also denote
\[
|B|^{2}_{w}:=w^{ij} B_{i}B_{j}.
\]
 Finally, if $F\in C^{2}(\Gamma_{+})$ is concave, then $F$ is also
concave as a curvature function depending on $\{h^{i}_{j}\}$.

We note some important properties of $H_{m}$ (see \cite{CS10} for a
simple derivation).
\begin{lemma}\label{property Hm}
Let $1\leq m \leq n$ be fixed, the following hold
\begin{itemize}
\item[i)] The $m$th roots $H_{m}^{1/m}$ are
concave in $\Gamma_{+}$.
\item[ii)] For $\forall i $, $\frac{\partial H_{m}}{\partial \lambda_{i}}(\lambda)>0$,
 where $\lambda \in \Gamma_{+}$.
 \item[iii)] $H_m^{1/m} \leq \ds\frac{H}{n}$; equivalently, $F \leq \left(\ds\frac{H}{n}\right)^{m \beta}$.
 \item[iv)]   $\tr(\dot F) \geq m \, \beta \, F^{1 -  \frac{1}{m\beta}}$.
\end{itemize}
\end{lemma}
The following algebraic property proved by Schulze in (\cite{Sch06},
Lemma\,2.5) will be needed in the last section.

\begin{lemma}\label{important inequality}
For any $\varepsilon> 0$ assume that $\lambda_{i} \geq \varepsilon H
> 0$, $i= 1, \ldots, n$, at some point of an $n$-dimensional
hypersurface. Then at the same point there exists a $\delta=
\delta(\varepsilon,n)> 0$ such that
\[
\frac{n\big|A\big|^{2}-H^{2}}{H^{2}}\geq \delta \left(
\frac{1}{n^n}-\frac{K}{H^n}\right).
\]
\end{lemma}

Consider the functions as in \cite{CR-M07}:
\begin{align*}
\s_{\kappa}(x)=\frac{\sinh(\sqrt{|\kappa|}x)}{\sqrt{|\kappa|}}&=\frac{\sinh(ax)}{a},&
\cc_{\kappa}(r)&=\s'_{\kappa}(x),\\
\ta_{\kappa}(x)&=\frac{\s_{\kappa}(x)}{\cc_{\kappa}(x)},&
\co_{\kappa}(x)&=\frac{1}{\ta_{\kappa}(x)}.
\end{align*}

Denote $r_{p}$ the function ``distance to $p$" in
${\mathbb{H}}_{\kappa}^{n+1}$ and use the notation
$\partial_{r_{p}}=\bar{\nabla}r_{p}$. And denote the component of
$\partial_{r_{p}}$ by $\partial^{\top}_{r_{p}}$ tangent to $M_{t}$,
which satisfies $\partial_{r_{p}}=\nabla(r_{p}|_{M^{n}})$. Define
the inner radius $\rho_{-}$ by 
\begin{align*}
\rho_{-}(t)&=\sup\{r: B_{r}(q)\ \text{is enclosed by}\ M_{t}\
\text{for some}\ q \in {\mathbb{H}}_{\kappa}^{n+1}\}
\end{align*}
where $B_{r}(q)$ is the geodesic ball of radius $r$ with centered at
$q$. The following well-known result for $h$-convex hypersurfaces in
${\mathbb{H}}_{\kappa}^{n+1}$ will be applied in later sections.
\begin{lemma}\label{$h$-convex}
Let $\Omega$ be a compact $h$-convex domain, $o$ the center of an
inball of $\Omega$, $\rho_{-}$ its inner radius. 
Furthermore let
$\tau:=\ta_{\kappa}(\frac{a\rho_{-}}{2})$, then
\begin{itemize}
\item[i)] the maximal distance $\max d(o, \partial \Omega)$ between $o$ and
the points in $\partial \Omega$ satisfies the inequality
$$\maxd(o,\partial \Omega) \leq \rho_{-} + a\frac{\ln(1+\sqrt{\tau})}{1+\tau}<\rho_{-} +a\ln 2. $$
\item[ii)] For any interior point $p$ of $\Omega$, $\langle \nu,
\partial_{r_{p}}\rangle \geq a\ta_{\kappa}(\dist((p,\partial \Omega))$, where $\dist$ denotes the distance in
the ambient space ${\mathbb{H}}_{\kappa}^{n+1}$ .
 \end{itemize}
\end{lemma}
\begin{proof}
 See (\cite{Bor}, Theorem\,3.1) for the proof.
\end{proof}

Our analysis relies on many a priori estimates on the H\"older norms
of the solutions to elliptic and parabolic partial differential
equations in the Euclidean space and regularity issues, which leads
to existence for all time and the convergence. We know that, in the
case of a function depending on space and time, there is a suitable
definition of H\"older norm which is adapted to the purposes of
parabolic equations (see e.g. \cite{Lieb}). In addition to the
standard Schauder estimates for linear equations, we use in the
paper some more recent results which are collected here. The
estimates below hold for suitable classes of weak solutions; for the
sake of simplicity, we state them in the case of a smooth classical
solution, which is enough for our purposes. These are used, for
example, by  Cabezas-Rivas and Sinestrari \cite{CS10}.

Given $r>0$, we denote by $B_r$ the ball of radius $r>0$ in
$\mathbb{R}^n$ centered at the origin. First we recall a well known
result of Krylov and Safonov, which applies to linear parabolic
equations of the form
\begin{equation} \label{ec_par} \left(a^{ij}(x, t) D_{i} D_j  +
b^i(x, t) D_i + c(x, t) - \frac{\partial}{\partial t}\right) u = f
\end{equation}
in $B_r \times [0,T]$, for some $T>0$. We assume that $a^{ij} =
a^{ji}$ and that $a^{ij}$ is elliptic; that is, there exist two
constants $\lambda, \, \Lambda > 0$ such that
\begin{equation}
\label{SP} \lambda |v|^2 \leq a^{ij}(x,t)v_iv_j \leq \Lambda |v|^2
\end{equation}
for all $v \in \mathbb{R}^n$ and all $(x, t) \in B_r \times [0,T]$.
Then the following estimate holds \cite[Theorem 4.3]{KS}:

\begin{theorem} \label{krylov-safonov} Let $u \in C^2(B_r \times [0,T])$ be a solution of
\eqref{ec_par} with measurable coefficients, satisfy \eqref{SP} and
$$|b^i|, \, |c|  \leq K_1 \qquad \text{for all } i = 1, \ldots, n,$$
for some $K_1>0$. Then, for any  $0<r'<r$ and any $0<\delta <T$ we
have
$$\|u\|_{C^{\alpha}(B_{r '} \times [\delta, T])} \leq C\left(\|u\|_{C(B_r \times [0,T])} + \|f\|_{L_\infty(B_r \times [0,T])}\right)$$
for some constants $C > 0$ and $\alpha \in (0,1)$ depending on $n$,
$\lambda$, $\Lambda$, $K_1$, $r$, $r'$ and $\delta$. \end{theorem}

Next we cite a result of Caffarelli to obtain spatial $C^{2,\alpha}$
estimates for $u$.  We consider the second-order, fully nonlinear
elliptic equations
\begin{equation}
\label{fully nonlinear equations} G(D^2u(x),x)=f(x), \qquad x \in
B_r.
\end{equation} Here
$G:{\mathcal S} \times B_r \to \mathbb{R}$, where $\mathcal S$ is
the set of the symmetric $n \times n$ matrices. The nonlinear
operator $G$ is called elliptic if there exist $\Lambda \geq \lambda
>0$ such that
\begin{equation} \label{elliptic condition} \lambda ||B|| \leq G(A+B,x) - G(A,x)
\leq \Lambda ||B||
\end{equation} for any $x \in B_r$ and any pair $A,B \in \mathcal S$ such that $B$
is nonnegative definite.

\begin{theorem}  \label{Caffarelli} Let $u \in C^2(B_r)$ be a solution of
\eqref{fully nonlinear equations}, where $G$ is continuous and
satisfies \eqref{elliptic condition}. Suppose in addition that $G$
is concave with respect to $D^2u$ for any $x \in B_r$. Then there
exists $\bar \alpha \in \,(0,1)$ with the following property: if,
for some $K_2>0$ and $\alpha \in \,(0,\bar \alpha)\,$, we have that
$f \in C^\alpha(\Omega)$ and that
$$
G(A,x) - G(A,y) \leq K_2 |x-y|^{\alpha} (||A||+1), \qquad x,y \in
B_r, \  A \in \mathcal S
$$
then, for any $0<r'<r$, we have the estimate
$$\|u\|_{C^{2,\alpha}(B_{r'})} \leq C(||u||_{C(B_r)}+||f||_{C^\alpha(B_r)}+1)$$
where $C > 0$ only depends on $n$, $\lambda$, $\Lambda$, $K_2$, $r$
and $r'$. \end{theorem}

The above result follows from Theorem 3 in \cite{Caf89} (see also
Theorem 8.1 in \cite{CC}). It generalizes, by a perturbation method,
a priori estimates for the solutions of linear elliptic second-order
equations to the viscosity solutions, due to Evans and Krylov, about
equations with concave dependence on the hessian. In contrast with
Evans-Krylov result (see e.g. inequality (17.42) in \cite{GT}),
Theorem \ref{krylov-safonov} gives an estimate in terms of the
$C^\alpha$-norm of $f$ rather than the $C^2$-norm, and this is
essential for our purposes.

\section{Short time existence and evolution equations}

This section first consider short time existence for the initial
value problem \eqref{vpmthpcf}-\eqref{def barF}.
\begin{theorem}
Let $X_{0}:M^{n}\rightarrow {\mathbb{H}}_{\kappa}^{n+1}$ be a smooth
closed hypersurface with mean curvature strictly bounded from below
by $na$ everywhere. Then there exists a unique smooth solution
$X_{t}$ of problem \eqref{vpmthpcf}-\eqref{def barF}, defined on
some time interval $[0, T )$, with $T > 0$.
\end{theorem}
\begin{proof}
We can argue exactly as in \cite[Theorem 3.1]{CS10}, although the
assumptions on the initial hypersurface and the ambient background
space in that paper are different, the proof applies to our case as
well.
\end{proof}

Proceeding now exactly as in \cite{Hui84} we derive some evolution
equations  on $M_{t}$ from the basic equation
\eqref{vpmthpcf}-\eqref{def barF}.
\begin{proposition}
For the ambient space $N^{n+1}={\mathbb{H}}_{\kappa}^{n+1}$, on any
solution $M_{t}$ of \eqref{vpmthpcf}-\eqref{def barF} the following
hold: \allowdisplaybreaks
\begin{align}
\label{evmetric}
&\partial_{t}g=2(\bar{F}-F)A,\\
\label{evimetric}
&\partial_{t}g^{-1}=-2(\bar{F}-F)g^{-1}\mathscr{W},\\
\label{evnormal}
&\partial_{t}\nu=\mathrm{X_{*}}(\nabla F),\\
\label{evvolume}
&\partial_{t}(\dif \mu_{t})=(\bar{F}-F)H\dif \mu_{t},\\
\label{evsecf} &\partial_{t}A=\Delta_{\dot{F}} A
  +  \ddot{F}(\nabla\mathscr{W},\nabla\mathscr{W})
  +\big[ \tr_{\dot F}(A\mathscr{W} )+a^{2}\tr(\dot F)\big] \, A \\
 &\qquad\quad + \big[\bar{F} - (m \beta + 1) F\big] A \mathscr{W}
 + a^{2}\big[\bar{F} - (m \beta + 1) F\big]g,\notag\\
 \label{evweigar}
&\partial_{t}\mathscr{W}=\Delta_{\dot{F}} \mathscr{W}
  +  \ddot{F}(\nabla\mathscr{W},\nabla\mathscr{W})
  +\big[ \tr_{\dot F}(A\mathscr{W} )+a^{2}\tr(\dot F)\big] \, \mathscr{W} \\
 &\qquad\quad - \big[\bar{F}+ (m \beta - 1) F\big] \mathscr{W}^{2}
 + a^{2}\big[\bar{F} - (m \beta + 1) F\big]\Id.\notag
\end{align}
\end{proposition}
\begin{proof}
The first fourth evolution equations under
\eqref{vpmthpcf}-\eqref{def barF} follow from straightforward
computation as in $ \S 3$ of \cite{Hui84}, and valid in an arbitrary
Riemannian manifold.

The evoltution of $A$ can be calculated from the definition of $A$
\allowdisplaybreaks
\begin{align*}
\partial_{t}h_{ij}=& -\frac{\partial}{\partial
t}\langle\bar{\nabla}_{\mathrm{X_{*}}(\partial_{i})}\mathrm{X_{*}}(\partial_{j}),
\nu\rangle.\\
=&-\langle\bar{\nabla}_{\mathrm{X_{*}}(\partial_{t})}\bar{\nabla}_{\mathrm{X_{*}}(\partial_{i})}\mathrm{X_{*}}(\partial_{j})
, \nu\rangle
-\langle\bar{\nabla}_{\mathrm{X_{*}}(\partial_{i})}\mathrm{X_{*}}(\partial_{j}),
\frac{\partial \nu}{\partial t}\rangle.\\
=&-\langle\bar{\nabla}_{\mathrm{X_{*}}(\partial_{i})}\bar{\nabla}_{\mathrm{X_{*}}(\partial_{t})}\mathrm{X_{*}}(\partial_{j})
, \nu\rangle -\bar{R}\left(\mathrm{X_{*}}(\partial_{i}),
\mathrm{X_{*}}(\partial_{t}), \mathrm{X_{*}}(\partial_{j}),
\nu\right)\\
&-\langle\bar{\nabla}_{\mathrm{X_{*}}(\partial_{i})}\mathrm{X_{*}}(\partial_{j}),
\mathrm{X_{*}}(\nabla F)\rangle\\
=&-\langle\bar{\nabla}_{\mathrm{X_{*}}(\partial_{i})}\bar{\nabla}_{\mathrm{X_{*}}(\partial_{j})}((\bar{F}-F)\nu)
, \nu\rangle
+(F-\bar{F})\bar{R}_{i0j0}-\nabla_{\nabla_{\partial_{i}}\partial_{j}}
F\\
=&\,\partial_{i}\partial_{j}F-\nabla_{\partial_{i}}\partial_{j}F -
(F-\bar{F})\langle\bar{\nabla}_{\mathrm{X_{*}}(\partial_{i})}\mathrm{X_{*}}(\partial_{k}),
\nu\rangle h_{j}^{k}-(\bar{F}-F)\bar{R}_{i0j0} \\
=&\, \Hess_{\nabla}F (\partial_{i},\partial_{j})+
(\bar{F}-F)h_{ik}h_{j}^{k}-(\bar{F}-F)\bar{R}_{i0j0},
\end{align*}
where $\nu$ is arranged to be $e_{0}$. Note that the definition of
$\dot{F}$ and $\ddot{F}$ allow us to write $\Hess_{\nabla}F$ as
follows
\begin{align*}
\Hess_{\nabla}F (\partial_{i},\partial_{j})&=\nabla_{i}\nabla_{j}F
=\nabla_{i}(\dot{F}_{k}^{l}\nabla_{j}h_{l}^{k})\\
&=\dot{F}_{k}^{l}\nabla_{i}\nabla_{j}h_{l}^{k} +
\ddot{F}_{k}^{l}\,_{m}^{n} \nabla_{i}h^{m}_{n}\nabla_{j}h_{l}^{k}\\
&=\dot{F}^{kl}\nabla_{i}\nabla_{j}h_{kl} +
\ddot{F}_{k}^{l}\,_{m}^{n} \nabla_{i}h^{m}_{n}\nabla_{j}h_{l}^{k}.
\end{align*}
Recall a form of Simons' identity \cite{Sim} (a simple  derivation
can be found in \cite{SSY75}), which is a consequence of the
Gau{\ss} and Codazzi equations
\begin{align*}
\nabla_{i}\nabla_{j}h_{kl}=&\nabla_{k}\nabla_{l}h_{ij}+h_{ij}h_{kp}h^{p}_{l}
-h_{ip}h^{p}_{l}h_{kj}+h_{il}h_{kp}h^{p}_{j}-h_{ip}h^{p}_{j}h_{kl}\\
&+\bar{R}_{ikjp}h^{p}_{l}+ \bar{R}_{iklp}h^{p}_{j}
-\bar{R}_{plkj}h^{p}_{i}+\bar{R}_{pjli}h^{p}_{k}+\bar{R}_{0k0l}h_{ij}
-\bar{R}_{0i0j}h_{kl}\\
&+\bar{\nabla}_{i}\bar{R}_{0lkj}+\bar{\nabla}_{k}\bar{R}_{0jli}.
\end{align*}
 Therefore
\begin{align}\label{evsecf1}
\partial_{t}h_{ij}=&\,\dot{F}^{kl}\nabla_{k}\nabla_{l}h_{ij} +
\ddot{F}_{k}^{l}\,_{m}^{n} \nabla_{i}h^{m}_{n}\nabla_{j}h_{l}^{k}\\
&+\dot{F}^{kl}\big\{h_{ij}h_{kp}h^{p}_{l}
-h_{ip}h^{p}_{l}h_{kj}+h_{il}h_{kp}h^{p}_{j}-h_{ip}h^{p}_{j}h_{kl}\notag\\
&+\bar{R}_{ikjp}h^{p}_{l}+ \bar{R}_{iklp}h^{p}_{j}
-\bar{R}_{plkj}h^{p}_{i}+\bar{R}_{pjli}h^{p}_{k}+\bar{R}_{0k0l}h_{ij}
-\bar{R}_{0i0j}h_{kl}\notag\\
&+\bar{\nabla}_{i}\bar{R}_{0lkj}+\bar{\nabla}_{k}\bar{R}_{0jli}\big\}
+ (\bar{F}-F)h_{ik}h_{j}^{k}-(\bar{F}-F)\bar{R}_{i0j0}.\notag
\end{align}
Also note that in our case where the background space is a
hyperbolic space, the ambient space is locally symmetric
($\bar{\nabla}\bar{\R}=0$) and the Riemann curvature tensor takes
the form
\begin{equation}\label{hyperbolic curvature}
\bar{\R}_{\alpha\beta\gamma\delta}=-a^2\left(\bar{g}_{\alpha\gamma}\bar{g}_{\beta\delta}-\bar{g}_{\alpha\delta}\bar{g}_{\beta\gamma}\right).
\end{equation}
Since $F$ is a homogeneous function of the Weingarten map
$\mathscr{W}$ of degree $m\beta$, then \allowdisplaybreaks
\begin{equation}\label{F'D}
\dot{F}\mathscr{W}=m \beta F.
\end{equation}
Then, the relations \eqref{hyperbolic curvature} with
$\bar{\nabla}\bar{\R}=0$ and \eqref{F'D} apply to \eqref{evsecf1} to
give: \allowdisplaybreaks
\begin{align*}
\partial_{t}h_{ij}=&\,\dot{F}^{kl}\nabla_{k}\nabla_{l}h_{ij} +
\ddot{F}_{k}^{l}\,_{m}^{n}
\nabla_{i}h^{m}_{n}\nabla_{j}h_{l}^{k}+\dot{F}^{l}_{k}h^{k}_{p}h_{l}^{p}h_{ij}+a^{2}\dot{F}^{k}_{k}h_{ij}\\
&+\big[\bar{F} - (m \beta + 1) F\big]h_{ik}h_{j}^{k} +
a^{2}\big[\bar{F} - (m \beta + 1) F\big]g_{ij}.
\end{align*}
Hence in compact notation we have \eqref{evsecf}.

Finally recalling that $\mathscr{W}=g^{-1}A$ we have
\begin{align*}
\partial_{t}\mathscr{W}=&\,g^{-1}\partial_{t}A +
\partial_{t}g^{-1}A\\
=&\,\Delta_{\dot{F}} \mathscr{W}
  +  \ddot{F}(\nabla\mathscr{W},\nabla\mathscr{W})
  +\big[ \tr_{\dot F}(A\mathscr{W} )+a^{2}\tr(\dot F)\big] \, \mathscr{W} \\
 & + \big[\bar{F}- (m \beta +1) F\big] \mathscr{W}^{2}
 + a^{2}\big[\bar{F} - (m \beta + 1) F\big]Id-2(\bar{F}-F)\mathscr{W}^{2}\\
=&\,\Delta_{\dot{F}} \mathscr{W}
  +  \ddot{F}(\nabla\mathscr{W},\nabla\mathscr{W})
  +\big[ \tr_{\dot F}(A\mathscr{W} )+a^{2}\tr(\dot F)\big] \, \mathscr{W} \\
 &- \big[\bar{F}+ (m \beta - 1) F\big] \mathscr{W}^{2}
 + a^{2}\big[\bar{F} - (m \beta + 1) F\big]\Id,
\end{align*}
which is \eqref{evweigar}.
\end{proof}

In the next theorem, we derive the evolution of any homogeneous
function of the Weingarten map $\mathscr{W}$ defined on an evolving
hypersurface $M_{t}$ of ${\mathbb{H}}_{\kappa}^{n+1}$ under the flow
\eqref{vpmthpcf}-\eqref{def barF}.
\begin{theorem}\label{evGD}
If $G$ is a homogeneous function of the Weingarten map $\mathscr{W}$
of degree $\alpha$, then the evolution equation of $G$ under the
flow \eqref{vpmthpcf}-\eqref{def barF} in
${\mathbb{H}}_{\kappa}^{n+1}$ is the following
\begin{align*}
\partial_{t}G=&\,\Delta_{\dot{F}} G
-\dot{F}\ddot{G}(\nabla\mathscr{W},\nabla\mathscr{W})
  + \dot{G} \ddot{F}(\nabla\mathscr{W},\nabla\mathscr{W})
  +\alpha\big[ \tr_{\dot F}(A\mathscr{W} )+a^{2}\tr(\dot F)\big] \, G\\
 &- \big[\bar{F}+ (m \beta - 1) F\big] \dot{G}\mathscr{W}^{2}
 + a^{2}\big[\bar{F} - (m \beta + 1) F\big]\tr(\dot G).
\end{align*}
\end{theorem}
\begin{proof}
The definition of $\dot{G}$ and $\ddot{G}$ allow us to write
$\Hess_{\nabla}G$ as follows
\begin{align*}
\Hess_{\nabla}G =\dot{G}\,\Hess_{\nabla}\mathscr{W} +
\ddot{G}(\nabla\mathscr{W} ,\nabla\mathscr{W}),
\end{align*}
which gives
\begin{align*}
\Delta_{\dot{F}} G =\dot{F}g^{-1}\,\Hess_{\nabla} G=
\dot{G}\,\Delta_{\dot{F}} \mathscr{W}
  +  \dot{F}\ddot{G}(\nabla \mathscr{W} ,\nabla
  \mathscr{W}).
\end{align*}
Therefore, by \eqref{evweigar}
\begin{align*}
\partial_{t}G= &\,\dot{G} \partial_{t}\mathscr{W}\\
=&\,\dot{G}\Delta_{\dot{F}} \mathscr{W}
  +  \dot{G}\ddot{F}(\nabla\mathscr{W},\nabla\mathscr{W})
  +\big[ \tr_{\dot F}(A\mathscr{W} )+a^{2}\tr(\dot F)\big] \, \dot{G}\mathscr{W} \\
 &- \big[\bar{F}+ (m \beta - 1) F\big]\dot{G} \mathscr{W}^{2}
 + a^{2}\big[\bar{F} - (m \beta + 1) F\big]\tr(\dot G)\\
=&\,\Delta_{\dot{F}} G
-\dot{F}\ddot{G}(\nabla\mathscr{W},\nabla\mathscr{W})
  + \dot{G} \ddot{F}(\nabla\mathscr{W},\nabla\mathscr{W})
  +\alpha\big[ \tr_{\dot F}(A\mathscr{W} )+a^{2}\tr(\dot F)\big] \, G\\
 &- \big[\bar{F}+ (m \beta - 1) F\big] \dot{G}\mathscr{W}^{2}
 + a^{2}\big[\bar{F} - (m \beta + 1) F\big]\tr(\dot G),
\end{align*}
where Euler's theorem $\dot{G}\mathscr{W}= \alpha \mathscr{W}$ is
used in the last line.
\end{proof}

An immediate application of the theorem above is to obtain the
evolving equations for $H$, and $F$.
\begin{proposition}
For the ambient space $N^{n+1}={\mathbb{H}}_{\kappa}^{n+1}$, on any
solution $M_{t}$ of \eqref{vpmthpcf}-\eqref{def barF} the following
hold:
\begin{align}\label{evH}
&\partial_{t}H = \Delta_{\dot{F}} H + \tr \big[\ddot{F}(\nabla
\mathscr{W}, \nabla\mathscr{W})\big]
- \big(\bar{F} + (m \beta -1) F\big) |A|^2\\
 &\qquad\quad + \big[\tr_{\dot
F} (A \mathscr{W})+ a^{2}\tr(\dot F)\big]\,H
 +n a^{2}\big[\bar{F} - (m \beta + 1) F\big]\notag,\\
  \label{evF}
&\partial_{t}F = \Delta_{\dot F} F + (F - \bar{F}) \,\big[\tr_{\dot
F} (A \mathscr{W}) - a^{2}\tr(\dot F)\big].
\end{align}
\end{proposition}

For the proof of the main theorem, as mentioned in the introduction,
it is convenient for us to define some suitable perturbations of the
second fundamental form. Define the turbulent second fundamental
form
\begin{equation*}
\tilde{h}_{ij}=h_{ij}-a g_{ij}.
\end{equation*}
Denote $\tilde{A}$ (resp. $\tilde{\mathscr{W}}$) the matrix whose
entries are $\tilde{h}_{ij}$ (resp. $\tilde{h}^{i}_{j}$). Then
$\tilde{\lambda}_{i}$ given by
\begin{equation*}
\tilde{\lambda}_{i}=\lambda_{i}-a,\ \ i \in 1,\ldots, n,
\end{equation*}
 are the eigenvalues of $\tilde{\mathscr{W}}$.
 Denote the elementary symmetric functions of the
$\tilde{\lambda}_{i}$ by $\tilde{E}_{r}, 1\leq r\leq n.$ From the
definition it follows that
\begin{align*}
\tilde{H}=\tr_{g}{\mathscr{\tilde{W}}}=\tilde{E}_{1}=\sum_{i=1}^{n}\tilde{\lambda}_{i}=H-na,\\
\bigl|\tilde{A}\bigr|^{2}=\tr_{g}({\tilde{\mathscr{W}}^{t}\tilde{\mathscr{W}}})=\sum_{i=1}^{n}\tilde{\lambda}^{2}_{i}
=\bigl|A\bigr|^{2}+na^2-2Ha,\\
\tilde{K}=\det{\mathscr{\tilde{W}}}=\det\{\tilde{h}^{i}_{j}\}=\prod_{i=1}^{n}\tilde{\lambda}_{i}.
\end{align*}
It is easy to check that
$$\nabla_{k}\tilde{h}_{ij}=\nabla_{k}h_{ij},$$
and therefore the Codazzi equation hold for
$\nabla_{k}\tilde{h}_{ij}$.

The following theorem is easily obtained from \eqref{evmetric},
\eqref{evimetric}, \eqref{evsecf} and \eqref{evweigar} by the
definitions of $\tilde{A}$ and $\mathscr{\tilde{W}}$.
\begin{theorem}
For the ambient space $N^{n+1}={\mathbb{H}}_{\kappa}^{n+1}$, on any
solution $M_{t}$ of \eqref{vpmthpcf}-\eqref{def barF} the following
hold \allowdisplaybreaks
\begin{align}
\label{evtursecf} \partial_{t}\tilde{A}=&\,\Delta_{\dot{F}}
\tilde{A}+  \ddot{F}(\nabla\tilde{\mathscr{W}}
,\nabla\tilde{\mathscr{W}})
  +\big[\bar{F} - (m \beta + 1) F\big] A \tilde{\mathscr{W}}\\
  &+a\big[(m \beta + 1) F - \bar{F}\big]
  \tilde{A}+\tr_{\dot F}(\tilde{A}\tilde{\mathscr{W}} )\, A .\notag\\
 \label{evturweigar}
\partial_{t}\tilde{\mathscr{W}}=&\,\Delta_{\dot{F}}
\tilde{\mathscr{W}}+  \ddot{F}(\nabla\tilde{\mathscr{W}}
,\nabla\tilde{\mathscr{W}})
  +\big[(1- m \beta) F-\bar{F} \big] \mathscr{W} \tilde{\mathscr{W}}\\
  &+a\big[(m \beta + 1) F - \bar{F}\big]\tilde{\mathscr{W}}+\tr_{\dot F}(\tilde{A}\tilde{\mathscr{W}} )\, \mathscr{W} .\notag
\end{align}
\end{theorem}

\begin{proof}
By \eqref{evmetric} and \eqref{evsecf} \allowdisplaybreaks
\begin{align*}
\partial_{t}\tilde{A}=&\,\partial_{t}A-a\,\partial_{t}g\\
=&\,\Delta_{\dot{F}} A
  +  \ddot{F}(\nabla\mathscr{W},\nabla\mathscr{W})
  +\big[ \tr_{\dot F}(A\mathscr{W} )+a^{2}\tr(\dot F)\big] \, A + \big[\bar{F} - (m \beta + 1) F\big] A \mathscr{W}
\\
 & + a^{2}\big[\bar{F}  - (m \beta + 1) F\big]g-2a(\bar{F}-F)A\\
 =&\,\Delta_{\dot{F}} \tilde{A}
  +  \ddot{F}(\nabla\tilde{\mathscr{W}} ,\nabla\tilde{\mathscr{W}})
  +\big[ \tr_{\dot F}(\tilde{A}\tilde{\mathscr{W}} )+2am\beta\,F\big] \, A
  + \big[\bar{F} - (m \beta + 1) F\big] A \tilde{\mathscr{W}}\\
 &+a\big[\bar{F} - (m \beta + 1) F\big] A
 + a^{2}\big[\bar{F} - (m \beta + 1) F\big]g-2a(\bar{F}-F)A\\
 =&\,\Delta_{\dot{F}} \tilde{A}
  +  \ddot{F}(\nabla\tilde{\mathscr{W}} ,\nabla\tilde{\mathscr{W}})
  +\big[\bar{F} - (m \beta + 1) F\big] A \tilde{\mathscr{W}}\\
  &+a\big[(m \beta + 1) F - \bar{F}\big]
  \tilde{A}+\tr_{\dot F}(\tilde{A}\tilde{\mathscr{W}} )\, A,
\end{align*}
where the third line follows by the relation
\begin{align*}
\tr_{\dot F}(A\mathscr{W} )= \tr_{\dot
F}(\tilde{A}\tilde{\mathscr{W}} )+2am\beta\,F-a^{2}\tr(\dot F).
\end{align*}
Then \eqref{evtursecf} and  \eqref{evimetric} together imply
 \eqref{evturweigar}.
\end{proof}

The evolution equation \eqref{evturweigar} of $\tilde{\mathscr{W}}$
applies to give the evolution of any homogeneous function of the
$\tilde{\mathscr{W}}$ defined on an evolving hypersurface $M_{t}$ of
${\mathbb{H}}_{\kappa}^{n+1}$ under the flow
\eqref{vpmthpcf}-\eqref{def barF}.
\begin{theorem}\label{evGD}
If $P$ is a homogeneous function of the turbulent Weingarten map
$\tilde{\mathscr{W}}$ of degree $\gamma$ , then the evolution
equation of $P$ under the flow \eqref{vpmthpcf}-\eqref{def barF} in
${\mathbb{H}}_{\kappa}^{n+1}$ is the following
\begin{align*}
\partial_{t}P=&\,\Delta_{\dot{F}} P
-\dot{F}\ddot{P}(\nabla\tilde{\mathscr{W}
},\nabla\tilde{\mathscr{W}})
  + \dot{P} \ddot{F}(\nabla\tilde{\mathscr{W}} ,\nabla\tilde{\mathscr{W}})
+ \big[(1- m \beta) F-\bar{F} \big] \dot{P}\tilde{\mathscr{W}}^{2}
 \\& +2a\,\gamma\big( F-\bar{F} \big) P
+\tr_{\dot F}(\tilde{A}\tilde{\mathscr{W}} ) \dot{P}\mathscr{W}.
\end{align*}
\end{theorem}

An immediate application of the theorem above is to obtain the
evolving equations for $\tilde{H}$, and $\tilde{H}^{n}$ and
$\tilde{K}$.
\begin{proposition}
For the ambient space $N^{n+1}={\mathbb{H}}_{\kappa}^{n+1}$, on any
solution $M_{t}$ of \eqref{vpmthpcf}-\eqref{def barF} the following
hold:
\begin{align}\label{evturH}
\partial_{t}\tilde{H}=& \Delta_{\dot{F}} \tilde{H} + \tr
\big[\ddot{F}(\nabla \tilde{\mathscr{W}}, \nabla
\tilde{\mathscr{W}})\big]
- \big(\bar{F} + (m \beta -1) F\big) |\tilde{A}|^2\\
 &+2a(F-\bar{F})\tilde{H}+\tr_{\dot F}(\tilde{A}\tilde{\mathscr{W}} )H\notag,\\
  \label{evturHn}\partial_{t}\tilde{H}^{n}=& \Delta_{\dot{F}} \tilde{H}^{n}
  -n(n-1)\tilde{H}^{n-2}|\tilde{H}|_{\dot{F}}^{2}
   + n\tilde{H}^{n-1}\tr\big[\ddot{F}(\nabla\tilde{\mathscr{W}},
\nabla \tilde{\mathscr{W}})\big]\\& + n\big( (1- m \beta)
F-\bar{F}\big)\tilde{H}^{n-1} |\tilde{A}|^2\
+2an(F-\bar{F})\tilde{H}^{n}\notag\\
&+n\,\tr_{\dot F}(\tilde{A}\tilde{\mathscr{W}} )\tilde{H}^{n}
+an^{2}\tr_{\dot F}(\tilde{A}\tilde{\mathscr{W}}
)\tilde{H}^{n-1}\notag,\\
\label{evturK}\partial_{t}\tilde{K}=&\,\Delta_{\dot{F}} \tilde{K}
-\dot{F}\ddot{\tilde{K}}(\nabla\tilde{\mathscr{W}}
,\nabla\tilde{\mathscr{W}})
  + \dot{\tilde{K}} \ddot{F}(\nabla\tilde{\mathscr{W}} ,\nabla\tilde{\mathscr{W}})
\\&
+ \big[(1- m \beta) F-\bar{F} \big]
\dot{\tilde{K}}\tilde{\mathscr{W}}^{2}
  +2a\,n\big( F-\bar{F} \big) \tilde{K}
+\tr_{\dot F}(\tilde{A}\tilde{\mathscr{W}} )
\dot{\tilde{K}}\mathscr{W}.\notag
\end{align}
\end{proposition}

Furthermore, \eqref{evturK} can be rewritten as
\begin{lemma}
For the ambient space $N^{n+1}={\mathbb{H}}_{\kappa}^{n+1}$, on any
solution $M_{t}$ of \eqref{vpmthpcf}-\eqref{def barF} the following
holds
\begin{align}
\label{evturK2}
\partial_{t}\tilde{K}=&\,\Delta_{\dot{F}} \tilde{K}
-\frac{(n-1)}{n}\frac{\left|\nabla \tilde{K}\right|_{\dot
F}^{2}}{\tilde{K}}
+\frac{\tilde{K}}{\tilde{H}^{2}}\left|\tilde{H}\nabla\tilde{\mathscr{W}} -\tilde{\mathscr{W}}\nabla\tilde{H}\right|^{2}_{\dot F,\tilde{b}}\\
&-\frac{\tilde{H}^{2n}}{n\tilde{K}}\left|\nabla(\tilde{K}\tilde{H}^{-n})\right|_{\dot
F}^2 +\tilde{K}\,\tr_{\tilde{b}}\left(
\ddot{F}(\nabla\tilde{\mathscr{W}} ,\nabla
\tilde{\mathscr{W}})\right) \notag
\\& + \big[(1- m \beta) F-\bar{F} \big]
\tilde{K}\tilde{H}
  +2a\,n\big( F-\bar{F} \big) \tilde{K}\notag\\
  &
+n\tr_{\dot F}(\tilde{A}\tilde{\mathscr{W}} ) \tilde{K}+a\tr_{\dot
F}(\tilde{A}\tilde{\mathscr{W}} ) \tilde{K}\tr(\tilde{b}),\notag
\end{align}
where $\tilde{b}:= \tilde{\mathscr{W}}^{-1}$.
\end{lemma}
\begin{proof}
Note that
\begin{equation}
\dot{\tilde{K}}=\tilde{K}\tilde{b},
\end{equation}
this implies
\begin{equation}\label{evtK}
\dot{\tilde{K}}\tilde{\mathscr{W}}^{2}=\tilde{K}\tilde{H},
\end{equation}
and
\begin{equation}\label{evtK1}
\dot{\tilde{K}} \ddot{F}(\nabla\tilde{\mathscr{W}}
,\nabla\tilde{\mathscr{W}})
=\tilde{K}\tilde{b}\ddot{F}(\nabla\tilde{\mathscr{W}}
,\nabla\tilde{\mathscr{W}})=\tilde{K}\,\tr_{\tilde{b}}\left(
\ddot{F}(\nabla\tilde{\mathscr{W}} ,\nabla
\tilde{\mathscr{W}})\right).
\end{equation}
A direct calculation as for example in Lemma $2.2$ of \cite{Cho85}
gives
\begin{equation}\label{evtK2}
-\dot{F}\ddot{\tilde{K}}(\nabla\tilde{\mathscr{W}}
,\nabla\tilde{\mathscr{W}}) =-\frac{\left|\nabla
\tilde{K}\right|_{\dot
F}^{2}}{\tilde{K}}-\tilde{K}\,\tr_{\tilde{b}}\left( \nabla\tilde{b}
\nabla \tilde{\mathscr{W}}\right)
\end{equation}
and
\begin{align}\label{evtK3}
-\tilde{K}\,\tr_{\tilde{b}}\left( \nabla\tilde{b} \nabla
\tilde{\mathscr{W}}\right)&=\frac{\tilde{K}}{\tilde{H}^{2}}\left|\tilde{H}\nabla\tilde{\mathscr{W}}
-\tilde{\mathscr{W}}\nabla\tilde{H}\right|^{2}_{\dot F,\tilde{b}}
+\frac{\left|\nabla \tilde{K}\right|_{\dot
F}^{2}}{n\tilde{K}}-\frac{\tilde{H}^{2n}}{n\tilde{K}}\left|\nabla(\tilde{K}\tilde{H}^{-n})\right|_{\dot
F}^2.
\end{align}
Therefore, identities \eqref{evtK}, \eqref{evtK1}, \eqref{evtK2} and
\eqref{evtK3} together apply to \eqref{evturK} to give
\eqref{evturK2}.
\end{proof}

\section{Preserving pingching}\label{Preserving pingching}

To control the pinching of the principal curvature along the flow
\eqref{vpmthpcf}-\eqref{def barF} of the Euclidean space, Schulze,
in \cite{Sch06}, following an idea of Tso \cite{Tso}, explored a
test function $Q=K/H^{n}$, which was also considered in \cite{CS10}.
An analogous quantity which is the quotient
$\tilde{Q}=\tilde{K}/\tilde{H}^{n}$ is more natural for our flow. By
the arithmetic-geometric mean inequality, $\tilde{Q}\leq 1/n^{n}$ on
$M_{t}$ and equality holds at a point in $M_{t}$ if and only if
$\tilde{\lambda}_{1}=\cdots=\tilde{\lambda}_{n}$, i.e,
$\lambda_{1}=\cdots=\lambda_{n}$ at the point. Thus, the only
hypersurfaces such that $\tilde{Q}= 1/n^{n}$ are the geodesic
spheres. The rest of this section consists of showing that the
inequality $\tilde{Q}\geq C^{*}$ for a suitable positive constant
$C^{*}$ remains under the evolution.

\begin{lemma}
For the ambient space $N^{n+1}={\mathbb{H}}_{\kappa}^{n+1}$, on any
solution $M_{t}$ of \eqref{vpmthpcf}-\eqref{def barF} the following
hold:
\begin{align}\label{evtildeq}
\partial_{t}\tilde{Q}=&\Delta_{\dot F} \tilde{Q}
+\frac{(n+1)}{n\tilde{H}^{n}}\left\langle\nabla
\tilde{Q},\nabla\tilde{H}^{n}\right\rangle_{\dot F}
-\frac{(n-1)}{n\tilde{K}}\left\langle\nabla \tilde{Q},\nabla\tilde
K\right\rangle_{\dot F}
-\frac{\tilde{H}^{n}}{n\tilde{K}}\left|\nabla\tilde{Q}\right|_{\dot
F}^2
\\
&+\frac{\tilde{Q}}{\tilde{H}^{2}}\left|\tilde{H}\nabla\tilde{\mathscr{W}}
-\tilde{\mathscr{W}}\nabla\tilde{H}\right|^{2}_{\dot F,\tilde{b}}
+\tilde{Q}\,\tr_{\tilde{b}-\frac{n}{\tilde{H}}\Id}\left(
\ddot{F}(\nabla\tilde{\mathscr{W}} ,\nabla \tilde{\mathscr{W}})\right)\notag\\
&+\big[( m \beta-1) F+\bar{F} \big]
\frac{\tilde{Q}}{\tilde{H}}\left(n\bigl|\tilde{A}\bigl|^{2}-\tilde{H}^{2}\right)
+a\tilde{Q}\tr_{\dot F}(\tilde{A}\tilde{\mathscr{W}}
)\left(\tr(\tilde{b})-\frac{n^{2}}{\tilde{H}}\right).\notag
\end{align}
\end{lemma}
\begin{proof}
By  \eqref{evturHn} and \eqref{evturK2}
\begin{align}
\partial_{t}\tilde{Q}&=\frac{1}{\tilde{H}^{n}}\partial_{t}\tilde{K}
-\frac{1}{\tilde{H}^{2n}}\partial_{t}\tilde{H}^{n}\notag\\
\label{evtildeQ}=&\frac{ \Delta_{\dot{F}} \tilde{K}}{\tilde{H}^{n}}
-\frac{\tilde{K}}{\tilde{H}^{2n}}\Delta_{\dot{F}} \tilde{H}^{n}
 -\frac{(n-1)}{n}\frac{\left|\nabla\tilde{K}\right|_{\dot{F}}^2}{\tilde{K}\tilde{H}^{n}}
-\frac{\tilde{Q}}{n}\left|\nabla\tilde{Q}\right|_{\dot{F}}^2
+n(n-1)\frac{\tilde{Q}}{\tilde{H}^{2}}\left|\nabla\tilde{H}\right|_{\dot{F}}^2
\\
&+\frac{\tilde{Q}}{\tilde{H}^{2}}\left|\tilde{H}\nabla\tilde{\mathscr{W}}
-\tilde{\mathscr{W}}\nabla\tilde{H}\right|^{2}_{\dot F,\tilde{b}}
+\tilde{Q}\,\tr_{\tilde{b}-\frac{n}{\tilde{H}}\Id}\left(
\ddot{F}(\nabla\tilde{\mathscr{W}} ,\nabla \tilde{\mathscr{W}})\right)\notag\\
&+\big[( m \beta-1) F+\bar{F} \big]
\frac{\tilde{Q}}{\tilde{H}}\left(n\bigl|\tilde{A}\bigl|^{2}-\tilde{H}^{2}\right)
+a\tilde{Q}\tr_{\dot F}(\tilde{A}\tilde{\mathscr{W}}
)\left(\tr(\tilde{b})-\frac{n^{2}}{\tilde{H}}\right).\notag
\end{align}
Furthermore, the first derivative and second derivative term in
\eqref{evtildeQ} can be computed as follows, the equality
\[
\nabla\left(\frac{\tilde{K}}{\tilde{H}^{n}}\right)
=\frac{\nabla\tilde{K}}{\tilde{H}^{n}}
-\frac{\tilde{K}}{\tilde{H}^{2n}}\nabla\tilde{H}^{n}
\]
implies
\begin{align}\label{secdriveQ}
\Delta_{\dot F}\left(\frac{\tilde{K}}{\tilde{H}^{n}}\right)
&=\frac{\Delta_{\dot F} \tilde{K}}{\tilde{H}^{n}}
-2\frac{\left\langle\nabla \tilde{H}^{n},\nabla\tilde
K\right\rangle_{\dot F}}{\tilde{H}^{2n}}
+2\frac{\tilde{K}}{\tilde{H}^{3n}}
\big|\nabla\tilde{H}^{n}\big|_{\dot F}^{2}
-\frac{\tilde{K}}{\tilde{H}^{2n}}\Delta_{\dot F} \tilde{H}^{n},
\end{align}
\begin{equation}\label{gradientQ1}
\begin{split}
\left\langle\nabla\left(\frac{\tilde{K}}{\tilde{H}^{n}}\right),\nabla\tilde{H}^{n}
\right\rangle_{\dot F}=\frac{\left\langle\nabla
\tilde{H}^{n},\nabla\tilde K\right\rangle_{\dot F}}{\tilde{H}^{n}}
-2\frac{\tilde{K}}{\tilde{H}^{2n}}\big|\nabla\tilde{H}^{n}\big|_{\dot
F}^{2},
\end{split}
\end{equation}
and
\begin{equation}\label{gradientQ2}
\begin{split}
\left\langle\nabla\left(\frac{\tilde{K}}{\tilde{H}^{n}}\right),\nabla\tilde{K}\right\rangle_{\dot
F} =\frac{\big|\nabla\tilde{K}\big|_{\dot F} ^{2}}{\tilde{H}^{n}}
-\frac{\tilde{K}}{\tilde{H}^{2n}}\left\langle\nabla
\tilde{H}^{n},\nabla\tilde K\right\rangle_{\dot F}.
\end{split}
\end{equation}
From \eqref{secdriveQ}, \eqref{gradientQ1} and \eqref{gradientQ2},
it follows
\begin{equation}\label{derivateterm}
\begin{split}
\frac{ \Delta_{\dot F}\tilde{K}}{\tilde{H}^{n}}&
-\frac{\tilde{K}}{\tilde{H}^{2n}}\Delta_{\dot F}\tilde{H}^{n}
 -\frac{(n-1)}{n}\frac{\left|\nabla\tilde{K}\right|_{\dot
F}^2}{\tilde{K}\tilde{H}^{n}}\\
&=\Delta_{\dot F} \left(\frac{\tilde{K}}{\tilde{H}^{n}}\right)
+\frac{(n+1)}{\tilde{H}^{n}}\left\langle\nabla
\left(\frac{\tilde{K}}{\tilde{H}^{n}}\right),\nabla\tilde{H}^{n}\right\rangle_{\dot
F}\\
&\quad-\frac{(n-1)}{n\tilde{K}}\left\langle\nabla
\left(\frac{\tilde{K}}{\tilde{H}^{n}}\right),\nabla\tilde
K\right\rangle_{\dot F}
-n(n-1)\frac{\tilde{K}}{\tilde{H}^{n+2}}\big|\nabla\tilde{H}^{n}\big|_{\dot
F}^{2}.
\end{split}
\end{equation}
Thus, the equation \eqref{derivateterm} applies to \eqref{evtildeQ}
to give \eqref{evtildeq}.
\end{proof}
In order to apply the maximum principle to \eqref{evtildeq} and show
that $\min_{p\in M_{t}}\tilde{Q}(p,t)$ is non-decreasing in time
some preliminary inequalities are needed in the sequel. The
following elementary property is a consequence of (\cite{CS10},
Lemma 4.2) (see also \cite{Cho85} and \cite{Sch06}).

\begin{lemma}\label{pingching imply convex lemma}
For any $\varepsilon \in (0, 1/n)$ and any $\tilde{\lambda} =
(\tilde{\lambda}_{1},\dots, \tilde{\lambda}_{n}) \in \mathbb{R}^n$
with $\tilde{\lambda}_{i} > 0$ for all $i = 1,\dots, n$, there
exists a constant $C = C(\varepsilon, n) \in  (0, 1/n^n)$  satisfies
\[
\tilde{Q}(\tilde{\lambda}) > C
\]
such that
\[
\tilde{\lambda}_{1}
>\varepsilon\tilde{H}(\tilde{\lambda}).
\]
\end{lemma}

The following estimate which is a stronger version of Lemma\,2.3
(ii) in \cite{Hui84} can be viewed as a generalisation by
Cabezas-Rivas and Miquel \cite{CS10}.
\begin{lemma}\label{a key estimate}
If $\tilde{H}>0$ and the inequality
$\tilde{\mathscr{W}}>\varepsilon\tilde{H}\,\Id$ is valid with some
$\varepsilon
>0$ at a point on a hypersurface immersed in ${\mathbb{H}}_{\kappa}^{n+1}$,
then $\varepsilon\leq 1/n$ and
\[
\left|\tilde{H}\nabla\tilde{\mathscr{W}}
-\tilde{\mathscr{W}}\nabla\tilde{H}\right|^{2} \geq
\frac{n-1}{2}\varepsilon^{2}\tilde{H}^{2}\left|\nabla\tilde{\mathscr{W}}\right|^{2}.
\]
\end{lemma}
\begin{proof}
The proof of the Lemma can be argued exactly as in (\cite{CS10},
Lemma\,4.1), only define $\tilde{\mathscr{W}}:=\mathscr{W}-a\,\Id$
at a point on a hypersurface immersed in
${\mathbb{H}}_{\kappa}^{n+1}$.
\end{proof}

Also as in \cite{CS10}, the preceding two lemmas allow us to prove
the pinching estimate for our flow, which is one of the key steps in
the proof of our main result.
\begin{theorem}
 \label{pinching} There exists a constant $C^{*}=C^{*}(a,n,m,\beta)
\in (0,1/n^n)$ with the following property: if $X:M^n \times [0,T)\,
\rightarrow {\mathbb{H}}_{\kappa}^{n+1}$, with $t \in [0,T)$, is a
smooth solution of \eqref{vpmthpcf}-\eqref{def barF}, with $F$ given
by \eqref{def Hm} for some $\beta \geq 1/m$, such that
\begin{itemize}
\item the initial immersion $X_0$ satisfies \eqref{ini pinching} with the constant $C^{*}$,
\item the solution $M_t=X(M^n,t)$ satisfies $\tilde{H}>0$ for all times $t \in [0,T)$,
\end{itemize}
then the minimum of $\tilde{K}/\tilde{H}^n$ on $M_t$ is
nondecreasing in time.
\end{theorem}
\begin{proof}
The assumption $\tilde{H}>0$ on evolving hypersurface ensures that
the quotient $\tilde{Q}$ is well-defined for $t \in [0,T)$. For the
proof of the Theorem, it is suffices to prove that the minimum of
$\tilde{Q}$ (denote by $\tilde{\mathcal {Q}}$) is nondecreasing in
time. First, by \eqref{ini pinching}, $\tilde{\lambda}_{1}>0$ on
$M_{t}$ for $t=0$, then this implies that $\tilde{\lambda}_{1}>0$ on
$M_{t}$ for $t \in [0,T)$ by a contradiction argument. In fact,
suppose to the contrary that there exists a first time $t_{0}>0$ at
which $\tilde{\lambda}_{1}=0$ as some point, then $\tilde{\mathcal
{Q}}(t_{0})=0$. On the other hand, by $\tilde{\lambda}_{1}>0$ on
$M_{t}$ for $t \in [0,t_{0})$, $\tilde{H}>0$ for all times $t \in
[0,t_{0})$. Thus applying the theorem on $[0,t_{0})$ implies that
$\tilde{\mathcal {Q}}(t)$ is nondecreasing in $[0,t_{0})$. So it
cannot decrease from $C^{*}$ to zero as $t$ goes to $t_{0}$, which
gives a contradiction. Now applying the maximum principle to
equation \eqref{evtildeq} for $\tilde{\mathcal {Q}}$ gives
\allowdisplaybreaks
\begin{align}
\partial_{t}\tilde{\mathcal {Q}}\geq
&\,\frac{\tilde{\mathcal
{Q}}}{\tilde{H}^{2}}\left|\tilde{H}\nabla\tilde{\mathscr{W}}
-\tilde{\mathscr{W}}\nabla\tilde{H}\right|^{2}_{\dot F,\tilde{b}}
+\tilde{\mathcal {Q}}\,\tr_{\tilde{b}-\frac{n}{\tilde{H}}\Id}\left(
\ddot{F}(\nabla\tilde{\mathscr{W}} ,\nabla \tilde{\mathscr{W}})\right)\notag\\
&+\big[( m \beta-1) F+\bar{F} \big] \frac{\tilde{\mathcal
{Q}}}{\tilde{H}}\left(n\bigl|\tilde{A}\bigl|^{2}-\tilde{H}^{2}\right)
+a\tilde{\mathcal {Q}}\tr_{\dot F}(\tilde{A}\tilde{\mathscr{W}}
)\left(\tr(\tilde{b})-\frac{n^{2}}{\tilde{H}}\right)\notag\\
\label{evtildeq estimate}\geq &\, \tilde{\mathcal
{Q}}\bigg\{\frac{1}{\tilde{H}^{2}}\left|\tilde{H}\nabla\tilde{\mathscr{W}}
-\tilde{\mathscr{W}}\nabla\tilde{H}\right|^{2}_{\dot F,\tilde{b}}
-\left|\tilde{b}-\frac{n}{\tilde{H}}\Id\right|\,
\left|\ddot{F}(\nabla\tilde{\mathscr{W}} ,\nabla \tilde{\mathscr{W}})\right|\\
&+\big[( m \beta-1) F+\bar{F} \big]
\frac{1}{\tilde{H}}\left(n\bigl|\tilde{A}\bigl|^{2}-\tilde{H}^{2}\right)
+a\tr_{\dot F}(\tilde{A}\tilde{\mathscr{W}}
)\left(\tr(\tilde{b})-\frac{n^{2}}{\tilde{H}}\right)\bigg\}.\notag
\end{align}
The various terms appearing here can be estimated as follows, as in
\cite [Theorem 4.3] {CS10}. The $h$-convexity of $M_{t}$ implies
that the third term of RHS in inequality \eqref{evtildeq estimate}
can be dropped with the strictly $h$-convexity on $M_{t}$. The last
term can also be dropped by the arithmetic-harmonic mean inequality,
\[
\sum_{i=1}^{n}{\tilde{b}}_{i}^{i}-\frac{n^2}{\tilde{H}}\geq 0
\] on $M_{t}$.
It remains to estimate the first two terms of RHS in inequality
\eqref{evtildeq estimate}, now proceeding exactly as in \cite{CS10},
\cite{Cho85} and \cite{Sch06}, choose orthonormal frame which
diagonalises $\tilde{\mathscr{W}}$ so that
\begin{align}\label{mintildeQ estimate}
\left|\tilde{H}\nabla\tilde{\mathscr{W}}
-\tilde{\mathscr{W}}\nabla\tilde{H}\right|^{2}_{\dot F,\tilde{b}}
&=\sum_{i,m,n}{\dot
F}^{i}\frac{1}{\tilde{\lambda}_{m}}\frac{1}{\tilde{\lambda}_{n}}
\left(\tilde{H}\nabla_{i}{\tilde{h}}_{m}^{n}
-{\tilde{h}}_{m}^{n}\nabla_{i}\tilde{H}\right)^{2}\\
&\geq\frac{1}{\tilde{H}^{2}}\sum_{i,m,n}{\dot F}^{i}
\left(\tilde{H}\nabla_{i}{\tilde{h}}_{m}^{n}
-{\tilde{h}}_{m}^{n}\nabla_{i}\tilde{H}\right)^{2}\notag
\end{align}
where $\tilde{\lambda}_{m}\leq \tilde{H}$ was used in the last
inequality by strictly $h$-convexity of $M_{t}$, i.e.,
$\tilde{\lambda}_{m}>0$ for any $m$. Now the property that each
$\dot{F}^i$ is positive in the interior of the positive cone can be
used. More precisely, for any $\varepsilon \in (0,1/n]$
$$
{\Xi}_\varepsilon:=\{ \lambda =(\lambda_1,\dots,\lambda_n) \in
\mathbb{R}^n ~:~ \min_{1 \leq i \leq n} \tilde{\lambda}_i \geq
\varepsilon(\tilde{\lambda}_1+\dots+\tilde{\lambda}_n) >0 \, \},
$$
$$
W_1(\varepsilon)=\min \{ {\dot F}^i (\lambda) ~:~  1 \leq i \leq n,
\leq n, \ {\lambda} \in {\Xi}_\varepsilon, | {\lambda}|=1 \}.
$$
By homogeneity of ${\dot F}^i$ with degree $m\beta-1$ and Lemma
\ref{property Hm} ii), exactly as in  the formula at the top of
p.453 of \cite {CS10}, the following inequality holds:
$$
{\dot F}^i( \lambda) \geq W_1(\varepsilon) |\lambda|^{m\beta-1},
\qquad {\lambda} \in {\Xi}_\varepsilon,
$$
where $W_1(\varepsilon)$ is an increasing positive function of
$\varepsilon$. This estimation, $h$-convexity of a hypersurface and
Lemma \ref{a key estimate} together imply that the inequality
\eqref{mintildeQ estimate} can be estimated as follows:
\begin{equation}\label{mintildeQ estimate2}
\begin{split}
\left|\tilde{H}\nabla\tilde{\mathscr{W}}
-\tilde{\mathscr{W}}\nabla\tilde{H}\right|^{2}_{\dot F,\tilde{b}}
\geq\frac{n-1}{2}W_1(\varepsilon)\varepsilon^{2}
|\mathscr{W}|^{m\beta-1}\left|\nabla\tilde{\mathscr{W}}\right|^{2},
\end{split}
\end{equation}
for some $\varepsilon \in (0, 1/n]$.

The term $\left| \ddot{F}(\nabla\tilde{\mathscr{W}} ,\nabla
\tilde{\mathscr{W}})\right|$ is smooth as long as
$\tilde{\lambda}_{i}>0$ for any $i$, homogeneous of degree
$m\beta-2$ in ${\lambda}_{i}$ and quadratic in
$\nabla\tilde{\mathscr{W}}$. The following estimation from above the
term $\left| \ddot{F}(\nabla\tilde{\mathscr{W}} ,\nabla
\tilde{\mathscr{W}})\right|$ can be derived as in \cite[inequatity
(4.7)]{CS10}: For any $\varepsilon \in (0, 1/n]$, there exists a
constant $W_2(\varepsilon)$ such that, at any point where
$\tilde{\mathscr{W}}\geq \varepsilon \tilde{H} \Id$,
\begin{equation}\label{secder F}
\left| \ddot{F}(\nabla\tilde{\mathscr{W}} ,\nabla
\tilde{\mathscr{W}})\right| \leq W_2(\varepsilon)
|\mathscr{W}|^{m\beta-2}\left|\nabla\tilde{\mathscr{W}}\right|^{2},
\end{equation}
where $W_2(\varepsilon)$ is decreasing in $\varepsilon$.

A next step is to show that
$\left|\tilde{b}-\frac{n}{\tilde{H}}\Id\right|$ is small if the
principal curvatures are pinched enough. It is clear that
\[
\left|\tilde{b}-\frac{n}{\tilde{H}}\Id\right| \leq
\sqrt{n}\max\left\{\left(\frac{1}{\tilde{\lambda}_{1}}
-\frac{n}{\tilde{H}}\right),\left(\frac{n}{\tilde{H}}
-\frac{1}{\tilde{\lambda}_{n}}\right) \right\}.
\]
Since for some $\varepsilon \in (0, 1/n]$
\begin{equation}\label{lower estimate on min.tilde curvature}
\tilde{\lambda}_{1}\geq \varepsilon \tilde{H},
\end{equation}
then
\begin{equation}\label{lower estimate on min.tilde curvature2}
\frac{1}{\tilde{\lambda}_{1}} -\frac{n}{\tilde{H}} \leq
\frac{1-\varepsilon n}{\varepsilon \tilde{H}}.
\end{equation}
On other hand, \eqref{lower estimate on min.tilde curvature} gives
\begin{equation}\label{estimate on max.tilde curvature}
\tilde{\lambda}_{n}\leq \left(1-(n-1)\varepsilon\right)\tilde{H}
\end{equation}
which implies that
\begin{equation}\label{estimate on max.tilde curvature2}
\frac{n}{\tilde{H}} -\frac{1}{\tilde{\lambda}_{n}} \leq
\frac{(n-1)\left(1-n\varepsilon\right)}{\tilde{H}\left(1-(n-1)\varepsilon\right)}.
\end{equation}
This combines with estimate \eqref{lower estimate on min.tilde
curvature2} to give
\begin{equation}\label{eq:2.27}
\left|\tilde{b}-\frac{n}{\tilde{H}}\Id\right| \leq
\frac{\mathscr{N}(\varepsilon)}{\tilde{H}},
\end{equation}
where
\begin{equation}
\mathscr{N}(\varepsilon)=\left\{ \begin{aligned}
&\frac{\sqrt{n}(1-\varepsilon n)}{\varepsilon }, & 0<\varepsilon\leq \frac{1}{2(n-1)}, \notag\\
&\frac{\sqrt{n}(n-1)\left(1-n\varepsilon\right)}{\left(1-(n-1)\varepsilon\right)},
& \frac{1}{2(n-1)}<\varepsilon <\frac{1}{n}.\notag
\end{aligned} \right.
\end{equation}
 Thus, the inequalities $\tilde{H}< H$,
$\big|H\big|^{2}\leq n\big|\mathscr{W}\big|^{2}$, estimations
\eqref{mintildeQ estimate}, \eqref{mintildeQ estimate2},
\eqref{secder F} and \eqref{eq:2.27} together give:

\begin{align}\label{the
first two estimation}
\frac{1}{\tilde{H}^{2}}\left|\tilde{H}\nabla\tilde{\mathscr{W}}
-\tilde{\mathscr{W}}\nabla\tilde{H}\right|^{2}_{\dot F,\tilde{b}}
&-\left|\tilde{b}-\frac{n}{\tilde{H}}\Id\right|\,
\left|\ddot{F}(\nabla\tilde{\mathscr{W}} ,\nabla \tilde{\mathscr{W}})\right|\\
 &\geq \frac{1}{\tilde{H}}
|\mathscr{W}|^{m\beta-2}\left|\nabla\tilde{\mathscr{W}}\right|^{2}
\left(\frac{(n-1)}{2\sqrt{n}}W_1(\varepsilon)\varepsilon^{2}
-W_2(\varepsilon)\mathscr{N}(\varepsilon)\right)\notag
 \end{align}
To achieve our purpose by application of the maximum principle, it
is necessary that
$\mathscr{N}^\prime(\varepsilon):=\left(\frac{(n-1)}{2\sqrt{n}}W_1(\varepsilon)\varepsilon^{2}
-W_2(\varepsilon)\mathscr{N}(\varepsilon)\right)$ is non-negative on
$M_{t}$. In fact, $\mathscr{N}(\varepsilon)$ is a strictly
decreasing function of $\varepsilon$; in addition,
$\mathscr{N}(\varepsilon)$ is arbitrarily large as $\varepsilon$
goes to zero and tends to zero as $\varepsilon$ goes to $1/n$ by its
definition, $W_1(\varepsilon)$ is increasing and $W_2(\varepsilon)$
is decreasing. Therefore, $\mathscr{N}^\prime(\varepsilon)$ is a
strictly increasing function of $\varepsilon$, it is negative as
$\varepsilon$ goes to zero and positive as $\varepsilon$ goes to
$1/n$. So there exists a unique value $\varepsilon_{0}\in (0, 1/n)$
such that
\begin{equation} \label{varepsilon0}
\mathscr{N}^\prime(\varepsilon_{0})=0.
 \end{equation}
By Lemma \ref{pingching imply convex lemma} there exists a constant
$C^{*}\in  (0, 1/n^n)$ satisfies $ \tilde{Q}(\tilde{\lambda}) >
C^{*} $ such that $ \tilde{\lambda}_{1}
>\varepsilon\tilde{H}(\tilde{\lambda})
$ with a $\varepsilon_{0}\in (0, 1/n)$ given by \eqref{varepsilon0}.
Thus, if $\tilde{Q}> C^{*}\geq 0$ everywhere on the initial
hypersurface, applying the maximum principle for $\tilde{Q}$ implies
that $\partial_{t}\tilde{\mathcal {Q}}\geq 0$, i.e.,
$\tilde{\mathcal {Q}}$ is nondecreasing in time. This guarantees
that $\tilde{Q}> C^{*}$ is preserved under the flow
\eqref{vpmthpcf}-\eqref{def barF} in ${\mathbb{H}}_{\kappa}^{n+1}$.
\end{proof}

Theorems asserts that inequality $\tilde{Q}> C^{*}$ holds for all $t
\in [0,T)$, furthermore, the definition of $C^{*}$ and Lemma
\ref{pingching imply convex lemma} together shows that
\begin{equation}\label{tildle lambda pinching}
\tilde{\lambda}_{i}\geq \varepsilon_{0} \tilde{H}\quad \text{on}\
M^{n} \times [0,T) \quad \text{for each}\ i,
\end{equation}
where $\varepsilon_{0}$ is given by \eqref{varepsilon0}, which
implies
\begin{equation}\label{lambda pinching}
\lambda_{i}\geq \varepsilon_{0} H\quad \text{on}\ M^{n} \times [0,T)
\quad \text{for each}\ i.
\end{equation}

\section{Upper bound on $F$}
In this section uniform bounds from above on the speed for the flow
and for the curvature of the hypersurface are derived, depending
only on the initial data. The bounds on curvatures together with the
estimates in the next section will imply the long time existence of
the flow by well-known arguments. In order to achieve this, the
method is to study the evolution under the flow
\eqref{vpmthpcf}-\eqref{def barF} of the function
\begin{equation}\label{Z}
Z_{t}=\frac{F}{\Phi-\epsilon}.
\end{equation}
Here $\Phi=\s_{\kappa}(r_{p})\langle\nu, \partial_{r_{p}}\rangle$,
which could be seen as ``support function" of $M^{n}$ in
${\mathbb{H}}_{\kappa}^{n+1}$, and $\epsilon$ is a constant to be
chosen later. The method used to obtain these bounds is very robust,
and applies to the Gau$\ss$ curvature flow in \cite{Tso}, the flow
with a general class of speeds in \cite{And94}, the
volume-preserving anisotropic mean curvature flow \cite{And01}, the
mixed volume preserving mean curvature flows in \cite{McC04}, the
mixed volume preserving curvature flow in \cite{McC05}, the volume
preserving mean curvature flow in the hyperbolic space in
\cite{CR-M07} and the volume preserving flow by powers of the $m$th
mean curvature in \cite{CS10}.

Given a function $f: \mathbb{R}\rightarrow \mathbb{R}$, $f(r_p)$
will mean $f\circ r_p$. An extension of \cite [Lemma 3]{CR-M07} will
be needed later.
\begin{lemma}\label{extension lemma} In ${\mathbb{H}}_{\kappa}^{n+1}$,
\begin{align}\label{basic formula in H}
\displaystyle &\langle\bar{\nabla}_X\partial_{r_p},Y\rangle =
\bar{\nabla}^2 r_p(X,Y) =
\begin{cases}  0 &\rm{ if}\ \ X=
\partial_{r_p}\\
  \co_{\kappa}(r_p)
\langle X , Y\rangle &\rm{if}\ \ \langle X,\partial_{r_p}\rangle=0
\end{cases} ,\\\label{basic formula2}
&\quad \bar{\Delta}_{\dot F} r_p = \tr(\dot F)\co_{\kappa}(r_p).
\end{align}
Moreover, if $f:\mathbb{R}\longrightarrow\mathbb{R}$ is a $C^2$
function,
\begin{align}\label{odeltafr}
&\bar\Delta_{\dot F} (f(r_p)) =  f''(r_p)|\partial_{r_p}|_{\dot
F}^{2} + f'(r_p)\ \bar\Delta_{\dot F} r_p.
\end{align}
And, for the restriction of $r_p$ to a hypersurface $M$ of
${\mathbb{H}}_{\kappa}^{n+1}$, one has \allowdisplaybreaks
\begin{align}\label{delta formula}
\Delta_{\dot F} r_p =& -\tr(\dot F)\langle \nu, \partial_{r_p}\rangle +  \co_{\kappa}(r_p) \left(\tr(\dot F)\ -   |\partial_{r_p}^\top|_{\dot F}^2\right). \\
 \label{delta f}
 \Delta_{\dot F} (f(r_p)) =&\, f''(r_p)\   |\partial_{r_p}^\top|_{\dot F}^2 + f'(r_p)\   \Delta_{\dot F} r_p \\
 =&\,(  f''(r_p) - f'(r_p) \  \co_{\kappa}(r_p))
|\partial_{r_p}^\top|_{\dot
F}^2  \nonumber\\
&\,+ \ f'(r_p)\  ( \tr(\dot F) \ \co_{\kappa}(r_p) - \tr_{\dot
F}(\mathscr{W}) \langle \nu,
\partial_{r_p}\rangle). \nonumber
\end{align}
\end{lemma}
\begin{proof}
First \eqref{basic formula in H} and \eqref{basic formula2} follow
from \cite[page 46]{Pet98} (see also \cite{GW79}), and
\eqref{odeltafr} follows form a direct calculation. On the other
hand, the Gau{\ss} and Codazzi equations give the following
\begin{align*}
\Hess_{\bar{\nabla}}r_p(X,Y) & =
\bar{\nabla}^{2}r_p(X,Y)=\langle\bar{\nabla}_X\partial_{r_p},Y\rangle\\
&=\langle\bar{\nabla}_X \nabla r_p,Y\rangle +
A(X,Y)\left\langle\partial_{r_p}, \nu\right\rangle\\
&=\nabla^{2}r_p(X,Y)+ A(X,Y)\left\langle\partial_{r_p}, \nu\right\rangle\\
&=\Hess_{\nabla}r_p(X,Y)+ A(X,Y)\left\langle\partial_{r_p},
\nu\right\rangle.
\end{align*}
This combines with \eqref{basic formula in H} and \eqref{basic
formula2} to give \eqref{delta formula}. \eqref{delta formula} gives
\eqref{delta f} by a direct calculation.
\end{proof}

\begin{corollary}
For $t \in [0, T)$ and any constant $\epsilon$, on any solution
$M_{t}$ of \eqref{vpmthpcf}-\eqref{def barF} in
${\mathbb{H}}_{\kappa}^{n+1}$, the following holds
\begin{align}\label{evZ}
\partial_{t}Z&=\Delta_{\dot F} Z
+\frac{2\left\langle\nabla Z,\nabla\Phi\right\rangle_{\dot
F}}{\Phi-\epsilon} - \frac{\bar{F}}{\Phi-\epsilon}\left(\tr_{\dot
F}(A\mathscr{W})- a^{2}\tr(\dot F)\right)
- \cc_{\kappa}(r)\frac{Z}{\Phi-\epsilon}\bar{F}\\
 &\quad -\epsilon\frac{Z}{\Phi-\epsilon}\tr_{\dot
F}(A\mathscr{W}) - a^{2}\tr(\dot F)Z +
(1+m\beta)\cc_{\kappa}(r)Z^{2}.\nonumber
 \end{align}
\end{corollary}
\begin{proof}
Using \eqref{vpmthpcf} and \eqref{basic formula in H} a direct
calculation gives
\begin{equation}\label{evs}
\bar{\nabla}_{t}(\s_{\kappa}(r_{p})\partial_{r_{p}})
=\cc_{\kappa}(r_{p})(\bar{F}-F)\nu,
\end{equation}
which implies that
\begin{equation}\label{partial t phi}
\partial_{t} \Phi
 =\s_{\kappa}(r_{p})\langle\partial_{r_{p}},\nabla F\rangle
 + \cc_{\kappa}(r_{p})(\bar{F}-F)
\end{equation}
by combining \eqref{evnormal}. On the other hand, a direct
calculation gives
\begin{equation}\label{Delta Phi}
\begin{split}
\Delta_{\dot F}\Phi &=\left\langle\nu,
\partial_{r_{p}}\right\rangle\Delta_{\dot F}\s_{\kappa}(r_{p})
+ 2\left\langle \nabla\s_{\kappa}(r_{p}), \nabla\langle\nu,
\partial_{r_{p}}\rangle\right\rangle_{\dot F}
+\s_{\kappa}(r_{p}) \Delta_{\dot F}\langle\nu,
\partial_{r_{p}}\rangle.
\end{split}
\end{equation}
Taking $f=\s_{\kappa}$ and using \eqref{delta f} give
\begin{equation}\label{delta s}
\Delta_{\dot F} \left(\s_{\kappa}(r_p)\right) =
-\frac{1}{\s_{\kappa}(r_p)} |\partial_{r_p}^\top|_{\dot F}^2 -
\cc_{\kappa}(r_p)\  \tr_{\dot F}(\mathscr{W})\langle\nu,
\partial_{r_p}\rangle + \tr(\dot
F)\frac{\cc_{\kappa}^2}{\s_{\kappa}}(r_p).
\end{equation}
We choose a frame $\{e_{i}\}$ at $p$ which is normal to $\nu$ and
tangent to $M_{t}$. With respect to this frame field, let
$\{e^{i}\}$ be the field of dual frames. Direct computations having
into account \eqref{basic formula in H} give
\begin{align}
&\langle\nabla \ \s_{\kappa}(r_p), \nabla \langle\partial_{r_p}, \nu \rangle\rangle_{\dot F} \label{grand term} \\
& \qquad = -\frac{\cc_{\kappa}(r_p)^2}{\s_{\kappa}}(r_p)
\left\langle\partial_{r_p}, \nu\right\rangle
|\partial_{r_p}^\top|_{\dot F}^2 + \cc_{\kappa}(r_p)\ {\dot F}_{j}^i
\ A(\partial_{r_p}^\top, \langle
\partial_{r_p}^\top, e^{j} \rangle e_{i}). \nonumber
\end{align}
Since
\begin{align}\label{delta1}
\Delta_{\dot F}\langle\nu,
\partial_{r_{p}}\rangle=\langle\nu,
\bar{\Delta}_{\dot
F}\partial_{r_{p}}\rangle+\langle\bar{\Delta}_{\dot F}\nu,
\partial_{r_{p}}\rangle +2 \langle \bar{\nabla}\nu,
\bar{\nabla}\partial_{r_{p}}\rangle_{\dot F},
\end{align}

\begin{align}\label{twotime nabla1}
\langle\nu, \bar{\nabla}_{i}\bar{\nabla}_{j}\partial_{r_{p}}\rangle
=&\frac{1}{\s_{\kappa}^2(r_p)}\langle\partial_{r_p},e_{i}\rangle\langle\partial_{r_p},e_{j}\rangle\langle\partial_{r_p},\nu\rangle
-\co_{\kappa}(r_p)h_{ij}\\
&- \co_{\kappa}^2(r_p)\ g_{ij}\ \langle\nu,\partial_{r_p}\rangle +
2\
\co_{\kappa}^2(r_p)\langle\partial_{r_p},e_{i}\rangle\langle\partial_{r_p},e_{j}\rangle\langle\partial_{r_p},\nu\rangle
\nonumber\\
&+\co_{\kappa}(r_p)\ h_{ij}
\langle\nu,\partial_{r_p}\rangle^2,\nonumber
\end{align}

\begin{align}\label{twotime nabla2}
\langle \bar{\nabla}_{j}\nu, \bar{\nabla}_{i}\partial_{r_{p}}\rangle
=\co_{\kappa}(r_p)\ h_{ij} - \co_{\kappa}(r_p)\
h(\partial_{r_p}^\top, \langle
\partial_{r_p}^\top, e_{j} \rangle e_{i}),
\end{align}

\begin{align}\label{product nabla}
\langle\bar{\nabla}_{i}\bar{\nabla}_{j}\nu, \partial_{r_{p}}\rangle
=&\langle\partial_{r_p},e_{k}\rangle \ \bar{\nabla}_{k}(h_{ij})-
\langle\nu,\partial_{r_p}\rangle h_{i}^{k} h_{kj},
\end{align}
combination of \eqref{delta1}, \eqref{twotime nabla1},
\eqref{twotime nabla2} and \eqref{product nabla} together implies

\begin{align}\label{delta<>}
\Delta_{\dot F} \langle \partial_{r_p}, \nu \rangle & =
 \frac{1}{\s_{\kappa}^2(r_p)}\left\langle\partial_{r_p},\nu\right\rangle |\partial_{r_p}^\top|_{\dot F}^2\
+ \co_{\kappa}(r_p) \ \tr_{\dot F}(\mathscr{W})\\
&\quad  - \tr(\dot F)\ \co_{\kappa}^2(r_p)
\langle\partial_{r_p},\nu\rangle   + 2\ \co_{\kappa}^2(r_p)\
\langle\nu,\partial_{r_p}\rangle |\partial_{r_p}^\top|_{\dot F}^2 \nonumber\\
& \quad+ \co_{\kappa}(r_p) \langle\nu,\partial_{r_p}\rangle^2\
\tr_{\dot F}(\mathscr{W})  - 2\ \co_{\kappa}(r_p)\ {\dot F}_{j}^i \
A(\partial_{r_p}^\top, \langle
\partial_{r_p}^\top, e^{j} \rangle e_{i})   \nonumber \\
& \quad + \langle\partial_{r_p}^\top, \nabla F\rangle-
\left\langle\partial_{r_p}, \nu\right\rangle \ \tr_{\dot
F}(A\mathscr{W}). \nonumber
\end{align}
From \eqref{Delta Phi}, \eqref{delta s}, \eqref{grand term} and
\eqref{delta<>}, it follows
\begin{equation*}
\Delta_{\dot F}\Phi =\co_{\kappa}(r_p) \ \tr_{\dot
F}(\mathscr{W})+\s_{\kappa}(r_{p})\langle\partial_{r_{p}},\nabla
F\rangle - \Phi \ \tr_{\dot F}(A\mathscr{W}).
\end{equation*}
Combining this with \eqref{partial t phi} yields
\begin{equation}\label{ev phi}
\partial_{t} \Phi=\Delta_{\dot F}\Phi + \Phi \ \tr_{\dot F}(A\mathscr{W}) +\cc_{\kappa}(r_p)\left(\bar{F} -
F- \tr_{\dot F}(\mathscr{W})\right).
\end{equation}
From \eqref{ev phi}, \eqref{evF} and \eqref{Z}, it follows
\begin{align}\label{evZ2}
\partial_{t} Z&=\frac{1}{\Phi-\epsilon}
\left(\Delta_{\dot F} F + (F - \bar{F}) \,\big[\tr_{\dot
F} (A \mathscr{W}) - a^{2}\tr(\dot F)\big]\right)\\
&\quad-\frac{F}{(\Phi-\epsilon)^{2}} \left(\Delta_{\dot F}\Phi +
\Phi \ \tr_{\dot F}(A\mathscr{W}) +\cc_{\kappa}(r_p)\left(\bar{F} -
F- \tr_{\dot F}(\mathscr{W})\right)\right).\notag
\end{align}
Another computation leads to
\begin{equation}\label{delta Z}
\begin{split}
\Delta_{\dot F} Z=\frac{\Delta_{\dot F}F}{\Phi-\epsilon} -\frac{
F\Delta_{\dot F}\Phi}{(\Phi-\epsilon)^{2}} -2\frac{1}{\Phi-\epsilon}
\langle\nabla Z,\nabla \Phi\rangle_{\dot F}.
\end{split}
\end{equation}
Replacing \eqref{delta Z} into \eqref{evZ2}, a few more computations
having into account $\tr_{\dot F}(\mathscr{W})=m\beta F$ by Euler's
theorem gives the desired evolution equation \eqref{evZ} of $Z$.
\end{proof}
In order to get a uniform  upper bound on $Z$, previously we have to
give the bounds on $r_{p}$ and $\left\langle\partial_{r_p}.
\nu\right\rangle$. The following estimate on $r_{p}$ for the
preserving volume mean curvature flow in \cite{CR-M07} is also valid
in our case with the help of Lemma \ref{$h$-convex} i).
\begin{lemma}
Let $\psi$ be the inverse of the function $\displaystyle s\mapsto
\vle(S^n) \int_0^s \s(\ell) d\ell$ and $\xi$ the inverse function of
$s\mapsto s+ \ds  a
\ln\frac{\left(1+\sqrt{\ta_{\kappa}(\frac{s}{2})}\right)^2}{1+\ta_{\kappa}(\frac{s}{2})}
$. If $V_0=\vle(\Omega_0)$ and ${\rho_{-}}(t)$ is the inner radius
of $\Omega_t$, then \begin{equation} \label{bound rho}
\xi(\psi(V_0)) \le {\rho_{-}}(t) \le \psi(V_0), \end{equation} for
every $t\in [0,T)$.
\end{lemma}

An immediate consequence of the lemma above and Lemma
\ref{$h$-convex} i) is

\begin{corollary}\label{distance estimate}
For every $t\in [0,T)$, if $p,q \in \Omega_t$, then
\begin{equation}\label{distance estimate2} \dist(p,q) < 2 (\psi(V_0) +  a \ \ln
2).\end{equation}
\end{corollary}

Now, if $p_{t_0}\in\Omega_t$ for an arbitrary fixed $t_0 \in [0,T)$,
then using \eqref{distance estimate2} gives an upper bound
$r_{p_{t_0}}(x) \le 2 (\psi(V_0)+  a \ \ln 2)$ for every $x\in M_t$.
Thus, for an upper bound on $F$, it is necessary to show that a
geodesic ball with fixed center remains inside the evolving
$\Omega_t$ for a short time.

\begin{lemma}\label{t0+smalltime}
If $B(p_{t_0},\rho_{t_0}) \subset \Omega_{t_0}$ for some $t_0 \in
[0,T)$, where $\rho_{t_0}={\rho_{-}}(t_0)$ is the inner radius of
$M_{t_0}$, then there exists some constant $\tau=\tau(a, n, m,
\beta, V_0)
>0$ such that $B(p_{t_0},\rho_{t_0}/2) \subset \Omega_t$ for every
$t \in [t_0, \min\{t_0+\tau, T\})$.
\end{lemma}
\begin{proof}
Proceeding similarly as in \cite[Lemma 8]{CR-M07}, our procedure is
to compare the deformation of $M_{t}$ by the equation
\eqref{vpmthpcf}-\eqref{def barF} with a geodesic sphere shrinking
under the $H_{m}^{\beta}$-flow.

For convenience, let $r_B(t)$ be the radius at time $t$ of a
geodesic sphere $\partial B(p_{t_0},r_B(t))$ centered at $p_{t_0}$,
evolving under $H_{m}^{\beta}$-flow and with the initial condition
$r_B(t_0) = \rho_{t_0}$. The radius of the evolving geodesic sphere
$\partial B(p_{t_0},r_B(t))$ satisfies
\begin{equation}\label{ev.rB}
\ds \frac{\dif r_B(t)}{\dif t} = - \
\co_{\kappa}^{m\beta} (r_B(t)).
\end{equation}
with the initial condition $r_B(t_0) = \rho_{t_0}$, this ODE has
solution
\begin{equation}\label{ODE solution}
\int_{\rho_{t_0}}^{r}\ta_{\kappa}^{m\beta}(s)\dif s=-(t-t_{0}).
\end{equation}
Denote $\mathcal {F}(r):=\int_{\rho_{t_0}}^{r}\ta_{\kappa}^{m\beta}
(s) \dif s$. Since $\mathcal {F}(r)$ is an increasing function in
$r$, then for $t\geq {t_0}$, $r_B(t) \geq \rho_{t_{0}}/2$ if and
only if
\[
t\, \leq t_{0} +
\int_{\rho_{t_0}/2}^{\rho_{t_0}}\ta_{\kappa}^{m\beta}(s)\dif s.
\]
On the other hand, let $\mathcal
{G}(s)=\int_{s/2}^{s}\ta_{\kappa}^{m\beta}(u)\dif u$, since
$s\mapsto \ta_{\kappa}(s)$ is increasing, then
\begin{equation*}
\frac{\dif\mathcal {G}(s)}{\dif s}>0
\end{equation*}
which shows that $\mathcal {G}(s)$ is increasing function in $s$.
Now using \eqref{bound rho}, this gives that if
\begin{equation}\label{short time}
t-t_{0} \ \leq \
\int_{\xi(\psi(V_0))/2}^{\xi(\psi(V_0))}\ta_{\kappa}^{m\beta}(s)\dif
s=:\tau,
\end{equation}
then
\begin{equation}\label{contained radius}
r_B(t) \ge \rho_{t_0}/2.
\end{equation}
For any $x\in M$, let $r(x,t)= r_{p_{t_0}}(X_t(x))$, from
\eqref{vpmthpcf}, it follows
\begin{equation} \label{evol.rm}
\frac{\dif r}{\dif t} = (\bar{F}(t)-F) \left\langle \nu_t,
\partial_{r_{p_{t_0}}}\right\rangle.
\end{equation}
If $\varphi:\mathbb{R}\rightarrow\mathbb{R}$ is a $C^{2}$ function,
set $f(x,t) = \varphi(r(x,t)) -\varphi(r_B(t))$, from \eqref{ev.rB}
and \eqref{evol.rm}, it follows
\begin{equation} \label{evol.f}
\partial_{t} f= \varphi'(r_{p_{t_0}}) \  (\bar{F}(t)-F) \left\langle \nu_t,
\partial_{r_{p_{t_0}}}\right\rangle + \varphi'(r_{B})\co^{m \beta}_{\kappa}(r_{B}).
\end{equation}
On the other hand, from \eqref{delta f}, it follows
\begin{align*}
\Delta f =&\,\Delta (\varphi(r_{p_{t_0}})) \\
=& \,(  \varphi''(r_{p_{t_0}}) - \varphi'(r_{p_{t_0}})\
\co_{\kappa}(r_{p_{t_0}}) |\partial_{r_p}^\top|^2  \nonumber\\
&+  \ \varphi'(r_{p_{t_0}})\  ( n\ \co_{\kappa}(r_{p_{t_0}}) - H
\langle \nu,
\partial_{r_{p_{t_0}}}\rangle). \nonumber
\end{align*}
Therefore, \eqref{evol.f} can be rewritten as
\begin{align} \label{evlove f}
\partial_{t} f =& \frac{F}{H}\Delta f +\varphi'(r_{p_{t_0}})\left\langle \nu_t,
\partial_{r_{p_{t_0}}}\right\rangle\bar{F}(t)+ \varphi'(r_{B})\co^{m \beta}_{\kappa}(r_{B})
\\
&- n\frac{F}{H}\varphi'(r_{p_{t_0}})\co_{\kappa}(r_{p_{t_0}})
+\frac{F}{H}\left[  \varphi'(r_{p_{t_0}})\ \co_{\kappa}(r_{p_{t_0}})
- \varphi''(r_{p_{t_0}})\right] |\partial_{r_p}^\top|^2. \nonumber
\end{align}
Taking $\varphi'(u)=\ta_{\kappa}(u)$ in \eqref{evlove f} gives
\begin{align} \label{evlove f2}
\partial_{t} f =& \frac{F}{H}\Delta f +\ta_{\kappa}(r_{p_{t_0}})\left\langle \nu_t,
\partial_{r_{p_{t_0}}}\right\rangle\bar{F}(t)+ \co^{m \beta - 1}_{\kappa}(r_{B})
\\
&- n\frac{F}{H}+\frac{F}{H}\left(  1 -
\frac{1}{\cc^{2}_{\kappa}}(r_{p_{t_0}})\right)
|\partial_{r_p}^\top|^2. \nonumber
\end{align}
Now, set $t_1= \inf\{t>t_0:\ p_{t_0}\notin\Omega_t\}$. Because
$\Omega_t$ is $h$-convex, Lemma \ref{$h$-convex} ii) implies
$\left\langle\nu_t, \partial_{r_{p_{t_0}}}\right\rangle \geq 0$ for
any $t \in [t_0, t_1]$. Thus, \eqref{evlove f2} combines with Lemma
\ref{property Hm} iii) and the initial condition to give
\begin{equation} \label{evlove inequality f2}
\left\{
\begin{array}{ll}
\partial_{t} f \geq  \frac{F}{H}\Delta f + \co^{m \beta - 1}_{\kappa}(r_{B})
- \left(\frac{H}{n}\right)^{m \beta - 1}, \\[2ex]
f(x,t_0) = \varphi(r(x,t_0)) -\varphi(\rho_{t_0})   \ge 0.
\end{array}\right.
\end{equation}
Next, set $\mathbbm{r}(t):=\min_{x\in M}r(x,t)$ for any $t \in [t_0,
t_1]$ and $\Theta(t):=\{x\in M \ | r(x,t)= \mathbbm{r}(t)\}$.
Applying  the minimum of $f$ to \eqref {evlove inequality f2} gives
\begin{equation} \label{evlove inequality f2}
\left\{
\begin{array}{ll}
\partial_{t} f_{\min} \geq  \co^{m \beta - 1}_{\kappa}(r_{B})
- \left(\frac{H_{\max}}{n}\right)^{m \beta - 1}, \\[2ex]
f_{\min}(t_0)  \ge 0.
\end{array}\right.
\end{equation}
 Note that any point where the minimum of $f$ is attained is the point
where the minimum of $r$ is attained for any $t \in [t_0, t_1]$, and
at the point the hypersurface is tangent to an inball of radius
$\mathbbm{r}(t)$, which implies that $H_{\max}=
n\co_{\kappa}(\mathbbm{r})$ on any point of $\Theta(t)$. Thus, using
a standard comparison principle concludes that
\begin{equation} \label{comparision}
f(x,t)\geq 0
\end{equation}
for any $t \in [t_0, t_1]$ as long as $f(x,t)$ is well defined for
$t \in [0,T)$, and it follows from \eqref{ODE solution} that
$r_{B}(t)$ is positive for $t \in [t_0, t_0 +
\int_{0}^{\xi(\psi(V_0))}\ta_{\kappa}^{m\beta}(s)\dif s ) [\supset
[t_0, t_0+\tau)]$. Then $f(x,t)\geq 0$  for any $t \in [t_0, \min\{
t_0+\tau, T, t_1\})$.

To complete the proof, assume that $t_1 < \min\{ t_0+\tau, T\}$. By
\eqref{comparision},
\[
r(x, t_1-\zeta) \geq r_B(t_1-\zeta) \ \ \text{for all}\ \zeta \in
(0, t_1-\tau].
\]
Hence by \eqref{contained radius},
\[
r(x, t_1)=\lim_{\zeta\rightarrow 0^{+}}r(x, t_1-\zeta) \geq
r_B(t_1)\geq \rho_{t_0}/2,
\]
which is a contradiction with $r(x, t_1)=r_{p_{0}}(t_1)=0$ by
definition of $t_{1}$. Therefore, $t_1 \geq\min\{ t_0+\tau, T\}$,
which, together with \eqref{comparision} and \eqref{contained
radius}, implies
\[
\mathbbm{r}(t) \geq \rho_{t_0}/2 \quad \text{on}\ [t_0, \min\{
t_0+\tau, T\}),
\]
which completes the proof.
\end{proof}

The above lemma assists us by allowing us to consider a uniform
bound on the speed of the flow.
\begin{theorem}
For $t \in [0,T)$,
\begin{equation}\label{bound on the speed}
F(\cdot,t) < C_{1}= C_{1}(n,m,\beta,a,M_{0}),
\end{equation}
moreover,
\begin{equation}\label{bound on the speed Hm}
H_{m}(\cdot,t) <C_{2}:= C^{1/\beta}_{1}.
\end{equation}
\end{theorem}
\begin{proof}
For any fixed $t_0 \in [0,T)$, let $p_{t_0}$ and $\rho_{t_{0}}$ be
as in Lemma \ref{t0+smalltime}. Then by Corollary \ref{distance
estimate} and Lemma \ref{t0+smalltime}, on the hypersurface $M_{t}$
for every $t \in [t_0, \min\{t_0+\tau, T\})$
\[
D_{1}:=\frac{\xi(\psi(V_0))}{2}\leq r_{p_{t_0}}\leq
\xi(\psi(V_0))=:D_{2}.
\]
Moreover, having into account Lemma \ref{$h$-convex} ii),
\[
\Phi=\s_{\kappa}(r_{p_{t_0}})\left\langle\nu,
\partial_{r_{p_{t_0}}}\right\rangle \geq
a\s_{\kappa}(D_{1})\ta_{\kappa}(D_{1}).
\]
Then, taking the constant $\epsilon =
a\s_{\kappa}(D_{1})\ta_{\kappa}(D_{1})/2$ leads to
\begin{equation}\label{good epsilon}
\Phi-\epsilon\geq\epsilon>0,
\end{equation}
which ensures $Z_{t}=\frac{\;F\;}{\Phi-\epsilon}$ is well-defined on
the same time interval.

Let us go back to the equation \eqref{evZ}, since strict
$h$-convexity holds for each $M_{t}$, $F$, $\bar{F}$ and $\tr_{\dot
F}(A\mathscr{W})- a^{2}\tr(\dot F)$ are all positive, which together
with \eqref{good epsilon}, the two terms containing $\bar{F}$ and
the term $a^{2}\tr(\dot F)Z $ can be neglected. Furthermore, note
that $F$ is homogeneous of degree $m\beta$, Euler's theorem and
\eqref{lambda pinching} together give the following
\[
\tr_{\dot F}(A\mathscr{W})=\dot F ^{i}\lambda^{2}_{i} \geq
\varepsilon_{0}H \dot F ^{i}\lambda_{i}=\varepsilon_{0}m\beta H  F.
\]
Now from the above remark,
\begin{align}\label{ev.inq.Z}
\partial_{t}Z&\leq \Delta_{\dot F} Z
+\frac{2\left\langle\nabla Z,\nabla\Phi\right\rangle_{\dot
F}}{\Phi-\epsilon} -\epsilon\varepsilon_{0}m\beta H Z^{2} +
(1+m\beta)\cc_{\kappa}(D_{2})Z^{2}.
 \end{align}
On the other hand, from \eqref{good epsilon} and Lemma \ref{property
Hm} iii), it follows
\[
Z\leq \frac{F}{\epsilon}\leq \frac{1}{\epsilon}\left(
\frac{H}{n}\right)^{m\beta}.
\]
Applying this to \eqref{ev.inq.Z} gives
\begin{align*}
\partial_{t}Z&
\leq \Delta_{\dot F} Z +\frac{2\left\langle\nabla
Z,\nabla\Phi\right\rangle_{\dot F}}{\Phi-\epsilon}\ +\left(
(1+m\beta)\cc_{\kappa}(D_{2})-\epsilon^{1+\frac{1}{m\beta}}nm\beta\varepsilon_{0}
Z^{\frac{1}{m\beta}}\right)Z^{2}.
 \end{align*}
Assume that in $(\bar{x},\bar{t})$, $\bar{t} \in [t_0,
\min\{t_0+\tau, T\})$,  $Z$ attains a big maximum $C\gg 0$ for the
first time. Then
\[
Z(\bar{x},\bar{t})\geq C(\Psi-\epsilon)(\bar{x},\bar{t})\geq
\epsilon C,
\]
which gives a contradiction if
\[
C> \max_{x\in
M^{n}}\left\{Z(x,t_0),\frac{1}{\epsilon}\left(\frac{\cc_{\kappa}(D_{2})(m\beta+1)}{n\varepsilon_{0}\epsilon
m\beta}\right)^{m\beta}\right\}.
\]
Thus,
\[
Z(x,t)\leq \max_{x\in
M^{n}}\left\{Z(x,t_0),\frac{1}{\epsilon}\left(\frac{\cc_{\kappa}(D_{2})(m\beta+1)}{n\varepsilon_{0}\epsilon
m\beta}\right)^{m\beta}\right\},
\]
on $[t_0, \min\{t_0+\tau, T\})$.

From the definition of $Z(x,t)$ and the upper bound $D_{2}$ of
$\rho_{t}$, it follows
\[
F(x,t)\leq \left(\s_{\kappa}(D_{2})-\epsilon\right)\max_{x\in
M^{n}}\left\{Z(x,t_0),\frac{1}{\epsilon}\left(\frac{\cc_{\kappa}(D_{2})(m\beta+1)}{n\varepsilon_{0}\epsilon
m\beta}\right)^{m\beta}\right\},
\]
on $[t_0, \min\{t_0+\tau, T\})$. Since $t_0$ is arbitrary, and
$\tau$ does not depend to $t_0$, this implies
\begin{align*}
F(x,t)\leq& \left(\s_{\kappa}(D_{2})-\epsilon\right)\max_{x\in
M^{n}}\left\{Z(x,t_0),\frac{1}{\epsilon}\left(\frac{\cc_{\kappa}(D_{2})(m\beta+1)}{n\varepsilon_{0}\epsilon
m\beta}\right)^{m\beta}\right\}\\
&=:C_{1}(n,m,\beta,a,M_{0})
\end{align*}
on $[0, T)$, which is \eqref{bound on the speed}, and so
\eqref{bound on the speed Hm} by the definition of $F$.
\end{proof}

Inserting the estimate \eqref{bound on the speed} into \eqref{ave
barF} immediately gives the following

\begin{corollary}
For $t \in [0,T)$,
\begin{equation}\label{bound on the speed barF}
\bar{F}(t) < C_{1}.
\end{equation}
\end{corollary}

Hence the speed of the evolving hypersurfaces is bounded.
\begin{corollary}
For $t \in [0,T)$,
\begin{equation}\label{bound on the speed partialtX}
\left|\frac{\partial}{\partial t}\mathrm{X}\left(p,t\right)\right| <
C_{3}:=2C_{1}.
\end{equation}
\end{corollary}

The curvature of $M_{t}$ also remains bounded.

\begin{corollary}
For $t \in [0,T)$,
\begin{equation}\label{bound on the speed H}
\big|\mathscr{W}\big|< H\leq C_{4}.
\end{equation}
\end{corollary}
\begin{proof}
The homogeneity of $F$, \eqref{lambda pinching} and the inequality
Lemma \ref{property Hm} iv) imply that

\begin{equation*}
m\beta F= \dot{F}\lambda_{i} \geq \varepsilon_{0}H \tr(\dot F) \geq
\varepsilon_{0}H  m \beta F^{1 -  \frac{1}{m\beta}}.
\end{equation*}
Thus, by \eqref{bound on the speed}
\[
H\leq \frac{1}{\varepsilon_{0}}F^{\frac{1}{m\beta}}\leq
\frac{1}{\varepsilon_{0}}C_{1}^{\frac{1}{m\beta}}=:C_{4},
\]
and so with the $h$-convexity of $M_{t}$
\[
\big|\mathscr{W}\big|<C_{4}.
\]
\end{proof}

\section{Long time existence}\label{Long time existence}
In this section, it will be shown that the solution of the initial
value problem \eqref{vpmthpcf}-\eqref{def barF} with the pinching
condition \eqref{ini pinching} exists for all positive times. As
usual, the first step is to obtain suitable bounds on the solution
on any finite time interval $[0, T)$, which guarantees the problem
\eqref{vpmthpcf}-\eqref{def barF} has a unique solution on the time
interval such that the solution converges to a smooth hypersurface
$M_{T}$ as $t\rightarrow T$. Thus, it is necessary to show that the
solution remains uniformly convex on the finite time interval which
ensures the parabolicity assumption of \eqref{vpmthpcf}-\eqref{def
barF}.

First it is to show the preserving $h$-convexity of the evolving
hypersurface $M_{t}$. Recall that Theorem \ref{pinching} and Lemma
\ref{pingching imply convex lemma} together imply the strictly
$h$-convexity of $M_{t}$. However, comparing with the initial
assumptions of Theorem \ref{main theorem}, there is a priori
assumption $\tilde{H}>0$ in Theorem \ref{pinching}. As Cabezas-Rivas
and Sinestrari mentioned in \cite{CS10}, note that for small times
such an assumption holds due to the smoothness of the flow for small
times and the initial pinching condition \eqref{ini pinching}, but
it is possible that at some positive time both $\min\tilde{K}$ and
$\min\tilde{H}$ tend to zero such that $\tilde{K}/\tilde{H}^n$
remains bounded. Thus, to exclude such a possibility, following
\cite{CS10}, it is necessary to complement Theorem \ref{pinching} by
establishing a positive lower bound on $\tilde{H}$ for the finite
time. \allowdisplaybreaks
\begin{lemma} \label{nonegative of tilde H} Under the hypotheses of Theorem \ref{pinching},
there exist $C_5,C_6>0$ depending on $n, m, \beta, a, M_0$ such that
\begin{equation} \label{tilde H_positive} \min_{M_t} \tilde{H} \geq C_5 e^{ - C_6 t } \qquad \forall
t \in [0, T). \end{equation} \end{lemma}
\begin{remark}
Here the lower bound on $\tilde{H}$ is enough for our purposes, and
we will give later a stronger lower bound on $\tilde{H}$ in Lemma
\ref {tilde H LowerBound}.
\end{remark}
\begin{proof}
Since under the hypotheses of Theorem \ref{pinching}, the evolving
hypersurfaces $M_t$ remains $h$-convex for every $t \in [0,T)$. Here
it is shown that $\tilde{H}$ satisfies the lower bound \eqref{tilde
H_positive}, which in particular implies that $\min \tilde{H}$
cannot go to zero as $t \to T$. Let us go back to the evolution
equation \eqref{evturH} of $\tilde{H}$,
\begin{align*}
\partial_{t}\tilde{H}=& \Delta_{\dot{F}} \tilde{H} + \tr
\big[\ddot{F}(\nabla\tilde{\mathscr{W}},
\nabla\tilde{\mathscr{W}})\big]
- \big(\bar{F} + (m \beta -1) F\big) |\tilde{A}|^2\\
 &+2a(F-\bar{F})\tilde{H}+\tr_{\dot F}(\tilde{A}\tilde{\mathscr{W}}
 )H.
\end{align*}
The various terms appearing here are easily estimated. First,
convexity of $F$ implies
\[
\tr \big[\ddot{F}(\nabla \tilde{\mathscr{W}},
\nabla\tilde{\mathscr{W}})\big]=\tr \big[\ddot{F}(\nabla\mathscr{W},
\nabla\mathscr{W})\big]\geq 0.
\]
 The next term can be estimated using the $h$-convexity of $M_t$, \eqref{bound on the speed} and
 \eqref{bound on the speed barF},
\begin{align*}
- \big(\bar{F} + (m \beta -1) F\big) |\tilde{A}|^2 \geq -
\big(\bar{F} +m \beta F\big) \tilde{H}^2 \geq - \big(1 +m \beta
\big)C_{1} \tilde{H}^2.
\end{align*}
A similar calculation applies to the third term,
\[
2a(F-\bar{F})\tilde{H}\geq -2a C_{1} \tilde{H}.
\]
The last term here is positive by the $h$-convexity of $M_t$.
Consequently,
\begin{align*}
\partial_{t}\tilde{H}&\geq \Delta_{\dot{F}} \tilde{H}
- \big(1 +m \beta \big)C_{1} \tilde{H}^2
 -2a C_{1} \tilde{H}\\
&\geq \Delta_{\dot{F}} \tilde{H} - \max\left\{\big(1 +m \beta
\big)C_{1} ,2a C_{1}\right\}(\tilde{H}^2
 +\tilde{H}).
\end{align*}
Now, if $\tilde{H} \geq 1$, then nothing is needed to prove. Thus,
if we assume $\tilde{H} < 1$ and set $C_{6}= 2\max\left\{\big(1 +m
\beta \big)C_{1} ,2a C_{1}\right\}$, then
\begin{align*}
\partial_{t}\tilde{H}\geq \Delta_{\dot{F}} \tilde{H}
 - C_{6}\tilde{H}.
\end{align*}
The parabolic maximum principle now gives $ \min_{M_t} \tilde{H}
\geq C_5 e^{ - C_6 t }$, where $C_5$ is given by $ \min_{M_0}
\tilde{H}$.
\end{proof}

\begin{corollary}
 \label{positive tildelambda} Let $X:M^{n} \times [0,T_{\max}) \, \rightarrow {\mathbb{H}}_{\kappa}^{n+1}$ be
the solution of \eqref{vpmthpcf}-\eqref{def barF} with an initial
value which satisfies the pinching condition \eqref{ini pinching}.
Then, the hypersurfaces $M_t$ are strictly $h$-convex on any finite
time interval; that is, for any $t\in [0,T)$, with $T < +\infty$ and
$T \leq T_{\max}$, we have
$$
\inf_{M^{n} \times [0,T)} \tilde{\lambda}_i >0, \qquad \forall
i=1,\dots,n.
$$
Therefore, Theorem \ref{pinching} is valid also without the
hypothesis that $\tilde{H}>0$ for $t \in (0,T)$. The same holds for
the other results that have been obtained until here under the same
assumptions of Theorem \ref{pinching}.
\end{corollary}
\begin{proof}
The conclusion follows the argument as in \cite[Corollary
6.2]{CS10}, only with obvious change of $\lambda_i$ by
$\tilde{\lambda}_i = \lambda_i - a$.
\end{proof}

Preserving $h$-convexity of the evolving hypersurface leads to the
following lower bound on the term $F$.
\begin{corollary}\label{the lower bound onF}
$F \geq a ^{m\beta}$ for all $t\in [0,T)$, with $T < +\infty$ and
$T \leq T_{\max}$.
\end{corollary}
\begin{proof}
In view of Lemma \ref{property
    Hm} iv), \eqref{bound on the speed} and \eqref{bound on the speed H}
    \begin{align}
        0 &\leq \left[\tr_{\dot
F} (A \mathscr{W}) - a^{2}\tr(\dot F)\right]=
        \sum_{i=1}^n F_i \left(\lambda_i^2 -a^2\right)\notag\\
        &\leq \left(C_{4}+ a\right) \sum_{i=1}^n \left( F_i \lambda_i  -  F_i a\right)
        \leq \left(C_{4}+ a\right) \left(m\beta F- m\beta F ^{1-\frac{1}{m\beta}}a\right)\notag\\
        &\leq m \beta C_{1}^{1-\frac{1}{m\beta}}\left(C_{4}+ a\right) \left(F^{\frac{1}{m\beta}}-
        a\right)\notag,
    \end{align}
    which implies
    \[
F - a^{m\beta}\geq 0.
    \]
\end{proof}

From Corollary \ref{the lower bound onF} we have the following lower
bound on $\bar{F}(t)$ by its definition.

\begin{corollary}\label{the lower bound on bar F}
$\bar{F}(t) \geq a ^{m\beta}$ for all $t\in [0,T)$, with $T <
+\infty$ and $T \leq T_{\max}$.
\end{corollary}

Since our flow is different from the volume preserving mean
curvature flow, we cannot follow the induction argument of Hamilton
as in \cite{CR-M07,Ham82,Hui84,Hui86,Hui87,McC03,McC04}, etc, to
obtain uniform estimates on all orders of curvature derivatives and
hence smoothness and convergence of the $M_{t}$ for the flow
\eqref{vpmthpcf}-\eqref{def barF}. Instead we use a more PDE
theoretic approach, following an argumentation similar to the one in
\cite{CS10}.

Before proceeding further, we adopt a local graph representation for
a $h$-convex hypersurface as in \cite{CR-M07}. For each fixed
$t_{0}$, let $p_{t_{0}}$ be a center of an inball of
$\Omega_{t_{0}}$, and $S^n$ the unit sphere in
$T_{p_{t_{0}}}{\mathbb{H}}_{\kappa}^{n+1}$. For each $t$, since
$M_{t}$ is $h$-convex, there exists a function $r:S^n \rightarrow
\mathbb{R}^+$ such that
 $M_{t}$ can be written as a map: $S^n \rightarrow {\mathbb{H}}_{\kappa}^{n+1}$,
 again denoted by $X_t$, satisfying
\begin{equation}
 \label{param1}
 X_t(x) = \exp_{p_{t_0}} r(t, u(t,x)) u(t,x),
\end{equation}
where $u(t,x) =
\ds\frac{\exp_{p_{t_0}}^{-1}X_t(x)}{r_{p_{t_0}}(X_t(x))}  \text{ and
}  r(t, u(t,x)) = r_{p_{t_0}}(X_t(x))$. At least, from Lemma
\ref{t0+smalltime} there exists some constant $\tau=\tau(a, n, m,
\beta, V_0)>0$ such that for $t \in [t_0, \min\{t_0+\tau, T\})$
(near $t_0$), $p_{t_0} \in \Omega_t$, and so the map $u_t: M^{n}
\rightarrow S^n\subset T_{p_{t_0}} {\mathbb{H}}_{\kappa}^{n+1}$
defined by $u_t(x)=u(t,x)$ is a diffeomorphism. On the other hand,
the map
\begin{equation}\label{another parametrization}
\breve X_t(x) = \exp_{p_{t_0}} r(t, u(t_0,x))\ u(t_0,x)
\end{equation} is another parametrization of $M_t$.
Incorporating a tangential diffeomorphism $\chi_{t} =
u^{-1}_{t_{0}}\circ u_t: M^{n} \rightarrow M^{n}$ into the flow
\eqref{vpmthpcf}-\eqref{def barF} to ensure that this
parametrization is preserved; that is, if $X_t$ is  a solution of
\eqref{vpmthpcf}-\eqref{def barF}, $\breve X_t$ satisfies the
equation
\begin{equation}
\label{evol.param} \left\langle {\partial_{t}\breve X_t} , \nu_t
\right\rangle = \bar{F}_t - F_t.
\end{equation}
 $\breve X_t$ can be
considered as a map from $S^n$ into ${\mathbb{H}}_{\kappa}^{n+1}$ by
using the diffeomorphism $u^{-1}_{t_{0}}$, i.e.,
\begin{equation}
\label{Xtparam on sn} \breve X_t(u) = \exp_{p_{t_0}} r(t,u)\ u \quad
\text { for every } u\in S^n,
\end{equation}
where $r(u)= r_{p_{t_0}}(\breve X_t(u))$ is a function on $S^n$. For
any local orthonormal frame $\{e_i\}$ of $S^n$,  let $D$ be the
Levi-Civita connection on $S^{n}$, a basis $\{\breve{e}_i \}$ of the
tangent space to $M_t$ is given  by
\begin{equation}\label{a basis breve ei}
\breve{e}_i = \breve X_{t\ast} e_i = D_i(r)
\partial_{r_{p_{t_0}}} + \s_{\kappa}(r) \tau_s e_i, \qquad 1\leq i
\leq n,
\end{equation}
where $\tau_s$ denotes the parallel transport along the geodesic
starting from $p_{t_0}$ in the direction of $u$, and until
$\exp_{p_{t_0}} r(u) u$. As in \cite{CR-M07}, by using Lemma
\ref{$h$-convex} and \eqref{bound rho} we deduce that
\begin{equation*}
 \left|\breve X_{t*} e_i\right|  <  \frac{\s_{\kappa}(\psi(V_0)+ a \ln 2)}{a\
 \ta_{\kappa}(\xi(\psi(V_0)))}.
\end{equation*}
Furthermore, \eqref{a basis breve ei} implies
\[
\big|e_i(r)\big| \leq \left|\breve X_{t*} e_i\right|.
\]
Therefore, both the first derivatives of $\breve X_t$ and $r$ are
bounded independently of $t$. The outward unit normal vector of
$M_{\hat{t}}$ can be expressed as
\begin{equation}\label{param outward unit normal}
\nu=\frac{1}{\left|\xi\right|}\Big(\s_{\kappa}(r)\partial_{r_{p}}-\sum_{i=1}^{n}D_{i}r{e}_{i}\Big)
\end{equation}
with
\begin{equation*}
\big|\xi\big|=\sqrt{\s^{2}_{\kappa}(r)+\left|Dr\right|^{2}}.
\end{equation*}
After a standard computation, the second fundamental form of $M_{t}$
can be expressed as
\begin{equation}\label{param h}
h_{ij}=-\frac{1}{\big|\xi\big|}\Big(\s_{\kappa}(r)D_{j}D_{i}r
-\s^{2}_{\kappa}(r)\cc_{\kappa}(r)\sigma_{ij}
-2\cc_{\kappa}(r)D_{i}rD_{j}r\Big),
\end{equation}
and the  metric $g_{ij}$ is
\begin{equation}\label{param g}
g_{ij}=D_{i}rD_{j}r+\s^{2}_{\kappa}(r)\sigma_{ij},
\end{equation}
where $\sigma_{ij}$ is the canonical metric of $S^n$. From this, the
inverse metric can be expressed as
\begin{equation}\label{param -g}
g^{ij}=\frac{1}{\s^{2}_{\kappa}(r)}\Big(\sigma^{ij}-\frac{1}{\big|\xi\big|^{2}}D^{i}rD^{j}r\Big),
\end{equation}
where $(\sigma^{ij})$=$(\sigma_{ij})^{-1}$ and
$D^{i}r=\sigma^{ij}D_{j}r$. Then equations \eqref{param h} and
\eqref{param -g} imply that
\begin{equation}\label{param weingarten}
h^{i}_{j}=-\frac{1}{\big|\xi\big|\s_{\kappa}(r)}
\Bigg[\frac{1}{\s_{\kappa}(r)}\Big(D_{j}D^{i}r-
\frac{D_{j}D_{l}rD^{i}rD^{l}r}{\big|\xi\big|^2}\Big)
-\cc_{\kappa}(r)\Big(\delta^{i}_{j}+\frac{D^{i}rD_{j}r}{\big|\xi\big|^{2}}\Big)\Bigg]
\end{equation}
and
\begin{equation}\label{param H}
H=-\frac{1}{\big|\xi\big|\s_{\kappa}(r)}\Big(\Delta_{\mathbb{S}}r
-\frac{1}{\big|\xi\big|^{2}}\nabla^{2}_{\mathbb{S}}r(Dr,Dr)\Big)
+\frac{\cc_{\kappa}(r)}{\big|\xi\big|}\Big(n+\frac{|Dr|^{2}}{\big|\xi\big|^{2}}\Big).
\end{equation}
Using \eqref{param g} and \eqref{param -g} the Christoffel symbols
have the expression:
\begin{align}\label{param christofel}
\Gamma^{k}_{ij}=\frac{1}{\s^{2}_{\kappa}(r)}\bigg[D_{i}D_{j}rD_{l}r+\s_{\kappa}(r)\cc_{\kappa}(r)
\big(D_{i}r\sigma_{lj}+D_{j}r\sigma_{il}-D_{l}r\sigma_{ij}\big)\bigg]\Big(\sigma^{kl}-\frac{1}{\big|\xi\big|^{2}}D^{k}rD^{l}r\Big).
\end{align}

\begin{lemma}
\label{UniformlyParabolic} Let $\phi:=H^{1/m}_{m}$  and
$$
\hat{\Gamma} =\{ \lambda=(\lambda_1,\dots,\lambda_n) ~:~ M_0 \leq
H(\lambda) \leq M_1, \quad \min_{1 \leq i \leq n}\lambda_i \geq
\varepsilon H(\lambda) \},
$$
which is a compact symmetric subset of the positive cone $\Gamma_+$.
There exists constants $ m_2 >m_1> 0$ depending only on $n, M_0$
such that for every $t \in [0, T)$ and $x \in M_t$, the following
inequality
$$
m_1 \leq \frac{\partial \phi}{\partial \lambda_i }(\lambda) \leq
m_2, \qquad i=1,\dots,n, \ \lambda \in \hat{\Gamma},
$$
holds as long as the hypersurfaces $M_t$ are strictly $h$-convex.
\end{lemma}
\begin{proof}
Since $\frac{\partial \phi}{\partial \lambda_i }(\lambda) > 0$ for
any $\lambda \in \hat{\Gamma}$, and $\hat{\Gamma}$ is compact, there
exist $m_2>m_1>0$ such that $$ m_1 \leq \frac{\partial
\phi}{\partial \lambda_i }(\lambda) \leq m_2, \qquad i=1,\dots,n, \
\lambda \in \hat{\Gamma}.
$$
\end{proof}

\begin{lemma}
 \label{appl_caff}
Let $M \subset {\mathbb{H}}_{\kappa}^{n+1}$ be an embedded
hypersurface satisfying at every point $D_3<H<D_4$, $\lambda_1 \geq
\varepsilon H$ for given positive constants $D_3,D_4,\varepsilon$.
Given any $p \in M$, let $r$ be a local graph representation of $M$
over a unit ball $S^{n} \subset T_pM$. Then $r$ satisfies
$$||r||_{C^{2,\alpha}(S^{n})} \leq C_{7}(1+||F||_{C^{\alpha}(S^{n})})$$
for some $C_{7}>0$ and $0<\alpha<1$ depending only on
$n,D_3,D_4,\varepsilon$ and the parameters $\beta,m$ in the
definition of $F$.
\end{lemma}
\begin{proof}
In exactly the same way as \cite[Lemma 6.3]{CS10}, let us set
$\phi:=H_m^{1/m}$. Recalling Lemma \ref{property Hm} i), $\phi$ is
concave in $\Gamma_+$. Then in this case, the Bellman's extension
$\bar{\phi}$ of $\phi$ takes the form
\[
\bar{\phi}(\bar{\lambda}):= \inf_{\lambda\in
\hat{\Gamma}}\left[\phi\left(\lambda\right)+
D\phi\left(\lambda\right)\left(\bar{\lambda}-\lambda\right) \right]
\]
for any $\bar{\lambda} \in \Gamma_+$. Notice that $\phi$ is
homogeneous of degree one, the extension simplifies to
\[
\bar{\phi}(\bar{\lambda})=\inf_{\lambda\in \hat{\Gamma}}
D\phi\left(\lambda\right)\bar{\lambda}.
\]
The Bellman extension preserves concavity, by definition, and
homogeneity, since it is the infimum of linear functions.
Importantly, $\bar{\phi}$  coincides with $\phi$ on $ \hat{\Gamma}$
by homogeneity of $\phi$. Furthermore, using the definition fo
$\bar{\phi}$ and Lemma \ref{UniformlyParabolic}, $\bar{\phi}$ is
uniformly elliptic, that
\[
m_1 |\bar{\eta}|\leq \bar{\phi}(\bar{\lambda}+\bar{\eta}) -
\bar{\phi}(\bar{\lambda})\leq \sqrt{n}m_2 |\bar{\eta}|, \quad
\mbox{for all}\; \bar{\lambda},\bar{\eta} \in {\mathbb{R}}^{n},\;
\bar{\eta}\geq 0.
\]

Now The hypotheses imply that the principal curvatures of $M$ at
every point are contained between two fixed positive constants. So
$M$ can be written as a local graph representation $r$ over a unit
ball $S^{n} \subset T_pM$ at a given point $p \in M$ with
$\|r\|_{C^2}$ bounded in terms of $D_4$. Let us consider the
function $\bar{\phi}(\bar{\lambda}(u))$, where $\bar{\lambda}(u)$
are the principal curvatures of $M$ at the point $(u,r(u))$. Since
$\bar{\lambda}(u)$ are the eigenvalues of a matrix depending on
$Dr,D^2r$ in view of \eqref{param weingarten},
$\bar{\phi}(\bar{\lambda}(u))$ can be expressed as
$\bar{\Phi}(Dr(u),D^2r(u))$ for a suitable function $ \bar{\Phi} =
\bar{\Phi}(u,A)$, with $(u,A) \in  S^n\times \mathcal {S}$,
$\mathcal {S}$ being the set of symmetric $n \times n$ matrices. The
dependence of $\bar{\Phi}$ on $A$ is related to the dependence of
$\bar{\Phi}$ on $\bar{\lambda}$. In fact, it is well known that the
concavity of $\bar{\Phi}$ with respect to $\bar{\lambda}$ implies
the concavity of $\bar{\Phi}$ with respect to $D^2 r$, and that
ellipticity on $\bar{\phi}$ implies the ellipticity condition
(\ref{elliptic condition}) for $\bar \Phi$. In addition, $\bar \Phi$
is homogeneous of degree one with respect to $D^2r$. Furthermore,
set $G(u,A):=\bar{\Phi}(Dr(u),A)$ and $ f(u)=
\bar{\Phi}(Dr(u),D^2r(u))$. The above argument implies that the
elliptic equation for $r$
\[
G(D^2r(u),u)=f(u), \qquad u \in S^{n}, \; r\in C^2(S^{n})
\]
satisfies the conditions of Theorem \ref{Caffarelli}. This theorem
gives that there exists $\alpha \in (0,1)$ such that
$$||r||_{C^{2,\alpha}(S^{n})} \leq C(1+||f||_{C^{\alpha}(S^{n})})$$
where $C$ depends on $n,D_3,D_4,\varepsilon,\beta$.

Our assumptions say that $\bar{\lambda}(u)$ belongs for every $u$ to
the set $\hat{\Gamma}$ where $\bar \phi$ and $\bar \Phi$ coincide.
Thus, $f$ coincides with $H_m^{1/m}=F^{1/\beta m}$ at
$\bar{\lambda}(u)$. Observe that our assumptions; that is the
uniform bounds on the curvatures both from below and above imply
that $F$ is contained between two positive values depending only on
$n,D_3,D_4,\varepsilon,\beta$. Therefore $||F^{1/\beta
m}||_{C^\alpha}$ is estimated by $||F||_{C^\alpha}$ times a constant
depending only on these quantities.  we obtain Lemma
\ref{appl_caff}.
\end{proof}

\begin{theorem}
 \label{Calpha estimate}
 Let $M_t$ be a solution of \eqref{vpmthpcf}-\eqref{def barF},
defined on any finite time interval $[0,T)$, with initial condition
satisfying \eqref{ini pinching}. Then, for any $0<t_{0}<T$, $\alpha
\in (0,1)$ and every integer $k \geq 0$, there exist constant
$C_{8}$, depending on $n,m,\beta,a,M_{0}$ and
$C_{9}=C_{9}(n,m,\beta,a,M_{0},k)$ such that
\begin{align}
\label{F estimate C alpha}\|F\|_{C^{\alpha}(M^{n} \times (t_{0},T])} &\leq C_{8},\\
\label{r estimate C alpha}\|r\|_{C^{k}(M^{n} \times (t_{0},T])}
&\leq C_{9}.
\end{align}
\end{theorem}
\begin{proof}
The cases $k=0,1$ of \eqref{r estimate C alpha} follow the fact that
$r$ and its first order derivatives are uniformly bounded. Then, by
\eqref{bound on the speed H}, \eqref{param weingarten} implies that
its second order derivatives are uniformly bounded. For each fixed
$t_{0}\in [0,T)$, $M_{t}$ can be locally reparameterized as graphs
over the unit sphere $\mathbb{S}^{n}$ with center $p_{t_{0}}$ in
$T_{p_{t_{0}}}{\mathbb{H}}_{\kappa}^{n+1}$ as \eqref{Xtparam on sn}.
Then, from  \eqref{evol.param}, \eqref{param outward unit
 normal} and \eqref{Xtparam on sn}, a short computation yields that the distance
 function $r$
on $\mathbb{S}^{n}$ satisfies the following parabolic PDE
\begin{align}\label{ev.param.r}
\partial_{t}r=\s^{-1}_{\kappa}(r)\big|\xi\big|(\bar{F}_t - F_t),
\end{align}
where the length of the outward normal vector $\big|\xi\big|$ is
given by the expression \eqref{param outward unit
 normal}. The function $F=H_{m}^{\beta}$ in the coordinate system
 under consideration is a function of $D^{2}r$ and $Dr$.
The RHS of \eqref{ev.param.r} is a fully nonlinear operator,
furthermore, observe that \eqref{ev.param.r} can be rewritten as
\begin{align}\label{ev.param.r2}
\partial_{t}r &=-\s^{-1}_{\kappa}(r)\big|\xi\big|H_{m}^{\beta}+\s^{-1}_{\kappa}(r)\big|\xi\big|\bar{F}_t\\
&= \s^{-1}_{\kappa}(r)H_{m}^{\beta -
1/m}\left(\left(-\big|\xi\big|\right)^{m}H_{m}\right)^{1/m}+\s^{-1}_{\kappa}(r)\big|\xi\big|\bar{F}_t\nonumber.
\end{align}
Since $H_{m}^{\beta - 1/m}$ is bigger than $a^{m\beta -1}$ in view
of Corollary \ref{the lower bound onF} and bounded above by
$C_{2}^{\beta-\frac{1}{m}}$ coming from \eqref{bound on the speed
Hm}, and $r$ can be bounded independently of $t$, this implies that
$\s^{-1}_{\kappa}(r)H_{m}^{\beta - 1/m}$ are also uniformly H\"older
 continuous functions. Then, from \eqref{param weingarten} this ensures that \eqref{ev.param.r}
is a linear, strictly parabolic partial differential equation.
However, the higher order regularity does not follow by the general
theory of Krylov and Safonov \cite{KS} because the operator is not a
concave function of $D^2 r$. Here, we use instead the arguments in
\cite{CS10} who followed the procedure of \cite{And04,Tsai,McC05},
which consists of first proving regularity in space at a fixed time
and then regularity in time.

The first step is to derive $C^{\alpha}$-estimate \eqref{F estimate
C alpha} of $F$. In the coordinate system
 under consideration, \eqref{evF} can be rewritten as
 \begin{align}\label{locally ev.F}
 \partial_{t}F = a^{ij} D_i D_j F  + b^i D_i F + e \, F,
\qquad (u,t) \in S^n \times J,
 \end{align}
 where $J=[t_0, \min\{t_0+\tau, T\})$, and the the coefficients are given by
$$a^{ij} = \beta \, H_m^{\beta - 1}\, c^{ij},
\quad b^i = \beta\, H_m^{\beta - 1} \, c^{lj} \Gamma_{lj}^i $$ and
$$ e = \beta\,(F - \bar{F})\,H_m^{-1}\left(\tr_c
(A\mathscr{W})-a^{2}\tr(c)\right).$$ Here
$$c=\{c^{ij}\}=\left\{\frac{\partial
H_m}{\partial h_{i}^{k}}g^{kj}\right\}.$$ Since $a^{ij}$, $b^i$ and
$e \, F$ are uniformly bounded on the curvatures both from above  on
any finite time interval in view of \eqref{bound on the speed H},
\eqref{param weingarten}-\eqref{param christofel} and from below by
$h$-convexity of $M_{t}$, the equation \eqref{locally ev.F} is
uniformly parabolic with uniformly bounded coefficients. Then
applying Theorem \ref{krylov-safonov} to \eqref{locally ev.F} gives
that for any $0 <r' < 1$ and $J'= J-t_0$ there exist some constants
$D_{6}>0$ and $\alpha \in (0,1)$ depending on $n, m, \beta, a,
M_{0}$ such that
\begin{equation}\label{Calphaestimate F}
 \|F\|_{C^\alpha(B_{r'} \times J')} \leq D_{5}
\|F\|_{C^0(M^{n} \times [0,T))} \leq  D_{6}.
\end{equation}
Therefore, covering $M_{t}$ by finite many graphs on balls of radius
$r'$ can give \eqref{F estimate C alpha}.

Furthermore, for any fixed time $t\in [t_{0}, T)$, applying \eqref{F
estimate C alpha} to Lemma \ref{appl_caff} on $M_{t}$ implies that
$$
||r||_{C^{2,\alpha}(M_{t})} \leq D_{7}:=D_{7}(n,m,\beta,a,M_{0},k).
$$
From this estimate on $r(\cdot,t)$ for any fixed $t$, following the
procedrre of  \cite[\S 3.3, \S 3.4]{And04} or \cite[Theorem
2.4]{Tsai} to equation \eqref{ev.param.r}, a $C^{2,\alpha}$ estimate
for $r$ with respect to both space and time can be obtained. Once
$C^{2,\alpha}$ regularity is established, standard parabolic theory
yields uniform $C^k$ estimates \eqref{r estimate C alpha} for any
integer $k > 2$, which implies uniform $C^k$ estimates \eqref{r
estimate C alpha} for any integer $k\geq 0$ with the fact that $r$,
its first order derivatives and its second order derivatives are
uniformly bounded.
\end{proof}

\begin{theorem} \label{long time exist}
If $M^{n}$ is a closed $n$-dimensional smooth manifold and $M_0:
M^{n} \rightarrow {\mathbb{H}}_{\kappa}^{n+1}$ is an immersion
pinched in the sense of \eqref{ini pinching}, then the solution
$M_{t}$ of \eqref{vpmthpcf}-\eqref{def barF} with initial condition
$X_0 $, is smooth and everywhere satisfies \eqref{ini pinching} on
$[0, \infty)$.
\end{theorem}
\begin{proof}
As we have already seen, the preserving pinching condition
\eqref{ini pinching} and smoothness throughout the interval of
existence follows from Theorem \ref{pinching} and Lemma
\ref{nonegative of tilde H}.

On the other hand, it is clear from the expression \eqref{Xtparam on
sn} for $\breve X_t$ that all the higher order derivatives of
$\breve X_t$ are bounded if and only if the corresponding
derivatives of $r$ are bounded. Thus, uniform $C^k$ estimates
\eqref{r estimate C alpha} of $r$ implies that all the derivatives
of $\breve X_t$ are also uniformly bounded. So the smoothness of
$\breve X_t$ implies the same smoothness of the reparametrization
$X_t$ of $\breve X_t$ given by \eqref{param1}.

It remains to show that the interval of existence is infinite.
Suppose to the contrary there is a maximal finite time $T$ beyond
which the solution cannot be extended and we derive a contradiction.
Then the evolution equation \eqref{vpmthpcf}-\eqref{def barF}
implies that
\[
\|X(p,\sigma)-X(p,\tau)\|_{C^{0}(X_{0})}\leq
\int_{\tau}^{\sigma}\big| \bar{F} - F\big|\left(p,t\right)\dif t
\]
for $0\leq \tau \leq \sigma < T$. By \eqref{bound on the speed} and
\eqref{bound on the speed barF}, $X(\cdot,T)$ tends to a unique
continuous limit $X(\cdot,T)$ as $t\rightarrow T$. In order to
conclude that $X(\cdot,T)$ represents a hypersurface $M_{T}$, next
under this assumption and in view of the evolution equation \eqref
{evmetric} the induced metric $g$ remains comparable to a fix smooth
metric $\tilde{g}$ on $M^{n}$:
\[
\left|\frac{\partial}{\partial
t}\left(\frac{g(\xi,\xi)}{\tilde{g}(\xi,\xi)}\right)\right|
=\left|\frac{\partial_{t}
g(\xi,\xi)}{g(\xi,\xi)}\frac{g(\xi,\xi)}{\tilde{g}(\xi,\xi)}\right|
\leq 2\big| \bar{F} -
F\big||A|_{g}\frac{g(\xi,\xi)}{\tilde{g}(\xi,\xi)},
\]
for any non-zero vector $\xi\in T M^{n}$,  so that ratio of lengths
is controlled above and below by exponential functions of time, and
hence since the time interval is bounded, there exists a positive
constant $C$ such that
\[
\frac{1}{C}\tilde{g}\leq g\leq C\tilde{g}.
\]
Then the metrics $g(t)$ for all different times are equivalent, and
they converge as $t\rightarrow T$ uniformly to a positive definite
metric tensor $g(T)$ which is continuous and also equivalent by
following Hamilton's ideas in \cite{Ham82}. Therefore using the
smoothness of the hypersurfaces $M_{t}$ and interpolation,
\allowbreak
\begin{align*}
&\|X(p,\sigma)-X(p,\tau)\|_{C^{k}(X_{0})} \\
&\leq C \|X(p,\sigma)-X(p,\tau)\|^{1/2}_{C^{0}(X_{0})}
\|X(p,\sigma)-X(p,\tau)\|^{1/2}_{C^{2k}(X_{0})}\\
&\leq C |\sigma - \tau|^{1/2},
\end{align*}
so the sequence $\{X(t)\}$ is a Cauchy sequence in $C^{k}(X_{0})$
for any $k$. Therefore $M_{t}$ converge to a smooth limit
hypersurface $M_{T}$ which must be a compact embedded hypersurface
in ${\mathbb{H}}_{\kappa}^{n+1}$. Finally, applying the local
existence result, the solution $M_{T}$ can be extended for a short
time beyond $T$, since there is a solution with initial condition
$M_{T}$, contradicting the maximality of $T$. This completes the
proof of Theorem \ref{long time exist}.
\end{proof}

\section{Exponential convergence to a geodesic sphere}
Observe that, to finish the proof of Theorem \ref{main theorem}, it
remains to deal with the issues related to the convergence of the
flow \eqref{vpmthpcf}-\eqref{def barF}: It should be proved that
solutions of equation \eqref{vpmthpcf}-\eqref{def barF} with initial
conditions \eqref{ini pinching} converge, exponentially in the
$C^{\infty}$-topology, to a geodesic sphere in
${\mathbb{H}}_{\kappa}^{n+1}$ as $t$ approaches infinity.

The first step is to show that, if a smooth limit exists, it has to
be a geodesic sphere of ${\mathbb{H}}_{\kappa}^{n+1}$. To address
the first step, let
\[
f= \frac{1}{n^n}-\frac{\tilde{K}}{\tilde{H}^n}.
\]
and we will show that the principal curvature come close together
when time tends to infinity. Then as remarked in Section
\ref{Preserving pingching}, $f\geq 0$ with equality holding only at
umbilic points, which is the value assumed on a geodesic sphere of
${\mathbb{H}}_{\kappa}^{n+1}$. The following Lemma is an immediate
consequence of the evolution equation \eqref{evtildeq} of
$\tilde{Q}$.
\begin{lemma}
The quantity $f$ evolves under \eqref{vpmthpcf}-\eqref{def barF}
satisfying
\begin{align}\label{evtildef}
\partial_{t} f=&\Delta_{\dot F} f
+\frac{(n+1)}{n\tilde{H}^{n}}\left\langle\nabla
 f,\nabla\tilde{H}^{n}\right\rangle_{\dot F}
-\frac{(n-1)}{n\tilde{K}}\left\langle\nabla  f,\nabla\tilde
K\right\rangle_{\dot F}
-\frac{\tilde{H}^{n}}{n\tilde{K}}\left|\nabla f\right|_{\dot F}^2
\\
&-\Bigg(\frac{\tilde{Q}}{\tilde{H}^{2}}\left|\tilde{H}\nabla\tilde{\mathscr{W}}
-\tilde{\mathscr{W}}\nabla\tilde{H}\right|^{2}_{\dot F,\tilde{b}}
+\tilde{Q}\,\tr_{\tilde{b}-\frac{n}{\tilde{H}}\Id}\left(
\ddot{F}(\nabla\tilde{\mathscr{W}} ,\nabla \tilde{\mathscr{W}})\right)\Bigg)\notag\\
&-\big[( m \beta-1) F+\bar{F} \big]
\frac{\tilde{Q}}{\tilde{H}}\left(n\bigl|\tilde{A}\bigl|^{2}-\tilde{H}^{2}\right)
-a\tilde{Q}\tr_{\dot F}(\tilde{A}\tilde{\mathscr{W}}
)\left(\tr(\tilde{b})-\frac{n^{2}}{\tilde{H}}\right).\notag
\end{align}
\end{lemma}

\begin{corollary}
Under the conditions of Theorem \ref{main theorem},
\begin{align}\label{ev.inequa.f}
\partial_{t} f\leq &\Delta_{\dot F} f
+\frac{(n+1)}{n\tilde{H}^{n}}\left\langle\nabla
 f,\nabla\tilde{H}^{n}\right\rangle_{\dot F}
-\frac{(n-1)}{n\tilde{K}}\left\langle\nabla  f,\nabla\tilde
K\right\rangle_{\dot F}
-\frac{\tilde{H}^{n}}{n\tilde{K}}\left|\nabla f\right|_{\dot F}^2
\\
&- C_{11}{\tilde{H}}f,\notag
\end{align}
where $C_{11}=\delta m\beta a^{m\beta}C^{*}$
\end{corollary}
\begin{proof}
Applying the similar argument as in Theorem \ref{pinching},
Corollary \ref{the lower bound on bar F}, inequality \eqref{ini
pinching} and Lemma \ref{important inequality} to \eqref{evtildef}
gives the conclusion.
\end{proof}

Thus, in order to apply the maximal principle to \eqref
{ev.inequa.f}, we have to obtain a stronger lower bound on
$\tilde{H}$. Firstly, the following corollary of the parabolic
Harnack inequality, which is proven in \cite[Proposition
6.1]{Mak12}, will allow us to estimate $\tilde{H}$ uniformly from
below.

\begin{proposition}
\label{HarnackIneq} For $(x_0, t_0) \in \mathbb{R}^n\times
\mathbb{R}$ and $\rho\in \mathbb{R}_+$ we let $G((x_0, t_0), \rho)
:= B_{\rho}(x_0) \times [t_0-\rho^2, t_0] \subset \mathbb{R}^n\times
\mathbb{R}$ and $G(\rho) := G((0,0), \rho)$. Let $\theta \in
C^{\infty}(G(4 R))$ be a nonnegative solution of an equation of the
form
\begin{equation}
\left(\partial_{t} -\Delta_{\dot F}\right)\theta = - \dot \theta +
a^{ij}\theta_{ij} + b^i \theta_i = g.
\end{equation}
Here $g=g(x, t, \theta(x, t))$ and we assume that there exists
$\zeta \in \mathbb{R}_+$ such that $f$ satisfies the inequality
$-\zeta \theta(x, t) \leq g(x, t, \theta(x, t)) \leq \zeta \theta(x,
t)$ for all $(x, t)\in G(4R)$. We assume the coefficients are
measureable and bounded by a constant $c_0 \in \mathbb{R}_+$ and
there exist
$0<\Lambda_{1}\leq \Lambda_{2} <\infty$ such that $\Lambda_{1} (\delta^{ij}) \leq (a^{ij}) \leq \Lambda_{2} (\delta^{ij})$. 
Then there exists $c=c(n, \Lambda_{1}, \Lambda_{2}, R, \Vert
b\Vert_{L^{\infty}}, c_0, \zeta) > 0$ such that there holds
\begin{equation}
    \underset{G\left((0, -4R^2), \frac R2\right)}{\sup} \theta \leq c\cdot \underset{G(R)}{\inf} \theta.
\end{equation}
\end{proposition}
From this the following lemma can be obtained.
\begin{lemma}
\label{tilde H LowerBound}
    There exists a constant $0< C_{12} = C_{12}(n,a,m,\beta,M_{0})$ such that for all $t\in [0, \infty)$
    \begin{equation}
       \tilde{H} \geq C_{12}.
    \end{equation}
\end{lemma}
\begin{proof}
    We will follow an idea from \cite[Lemma 6.2]{Mak12}. For $t\in [0, \infty)$
    let $x_t\in M_t$ be a point in contact with an enclosing sphere of radius $D_{1}<\rho < D_{2}$. This implies
    \begin{equation}
    \label{BoundFSupBelow}
        \underset{x\in M_t}{\sup} F(x) \geq F(x_t) \geq \co^{m\beta}_{\kappa} (\rho_{-}) \geq \co^{m\beta}_{\kappa}(D_{1}).
    \end{equation}
   In view of \eqref{evF}, $\theta:= F- a^{m\beta}$ satisfies the evolution equation
    \begin{equation}
        \left(\partial_{t} -\Delta_{\dot F}\right)(F- a^{m\beta}) = (F - \bar{F}) \,\big[\tr_{\dot
F} (A \mathscr{W}) - a^{2}\tr(\dot F)\big].
    \end{equation}
Since $ \left(\partial_{t} -\Delta_{\dot F}\right)$ is uniformly
parabolic in view of Corollary \ref{the lower bound onF} and Lemma
\ref{UniformlyParabolic},
    we can apply Proposition \ref{HarnackIneq} together with \eqref{BoundFSupBelow} to obtain
    the lower bound for $F - a^{m\beta}$ for all times
    \[
F - a^{m\beta} \geq c\left(\co^{m\beta}_{\kappa}(D_{1})-
a^{m\beta}\right).
    \]
    Therefore, with help of Lemma \ref{property Hm},
    we have
    \[
H \geq n\left[c\left(\co^{m\beta}_{\kappa}(D_{1})-
a^{m\beta}\right)+ a^{m\beta}\right]^{\frac{1}{m\beta}},
    \]
    which implies
    \[
\tilde{H} \geq C_{12},
    \]
    where \[
C_{12}= n\left\{\left[c\left(\co^{m\beta}_{\kappa}(D_{1})-
a^{m\beta}\right)+ a^{m\beta}\right]^{\frac{1}{m\beta}}- a\right\}>
0.
    \]
\end{proof}

\begin{proposition}\label{convergence of f}
Under the conditions of Theorem \ref{main theorem}, the rate of
convergence of $f$ to $0$ as $t\rightarrow \infty$ is exponential.
\end{proposition}
\begin{proof}
Applying the similar argument as in Theorem \ref{pinching} and Lemma
\ref{tilde H LowerBound} to \eqref{evtildef} gives
\begin{align}
\partial_{t} f_{\max}(t)\leq -C_{11}C_{12}f_{\max}(t).\notag
\end{align}
which implies that
\begin{align}
f_{\max}(t)\leq C_{13}\e^{- C_{14}t},\notag
\end{align}
where $C_{13}=f_{\max}(0)$, $C_{14}=C_{11}C_{12}$. This proves the
Proposition.
\end{proof}

As we have seen, since $M_{t}$ is strictly $h$-convex on $[0,
\infty)$, from Lemma \ref{t0+smalltime} there exists some constant
$\tau=\tau(a, n, m, \beta, V_0) (>0)$ such that for each $t_0 \in
[0, \infty)$, on the time interval $[t_0, t_0+\tau)$, $p_{t_0} \in
\Omega_t$, and $M_{t}$ can be represented as $\graph\,(r)$
\eqref{param1}. Let us go back to the evolution equation
\eqref{ev.param.r2} of $r$, we know that it is uniformly parabolic
because of $h$-convexity of $M_{t}$ for any $t\in [0, \infty)$. So
repeating the same arguments as in Section \ref{Long time existence}
implies that all the derivatives of $r$ are uniformly bounded. Thus,
we are now in conditions to apply Arzel-Ascoli Theorem to conclude
the existence of sequences of maps $r_{t_{i}}$ and $\breve
X_{t_{i}}$ solves \eqref{another parametrization} with
$t_{i}\rightarrow \infty$, which $C^{\infty}$-converge to smooth
maps $ r_\infty : S^n\rightarrow \mathbb{R}_+$ and $\breve X_\infty:
S^n \rightarrow {\mathbb{H}}_{\kappa}^{n+1}$ satisfying $
\check{X}_\infty(u) = \exp_{p_0} r_\infty(u) u$. Therefore, the
reparametrization $X_t$ of $\breve X_t$ given by \eqref{param1} has
the same convergence properties. Thus we conclude that $X_{t_{i}}$
$C^{\infty}$-converges to a map $X_\infty: S^n \rightarrow
{\mathbb{H}}_{\kappa}^{n+1}$, and that two immersions $ X_\infty$
and $\breve X_\infty$ are equivalent in the sense of an isometry of
${\mathbb{H}}_{\kappa}^{n+1}$ carrying $p_{\infty}$ to $p_0$. And
 since the convergence is smooth and
all the hypersurfaces $X_t(S^n)$ satisfy \eqref{ini pinching}, we
can assure that $ \mathscr{S} = X_\infty(S^n)$ must be a compact
embedded hypersurface and satisfies \eqref{ini pinching}.

On the other hand, Proposition \ref{convergence of f} says that all
points on $\mathscr{S}$ are umbilic points. In conclusion, the only
possibility is that $\mathscr{S}$ represents a geodesic sphere in
${\mathbb{H}}_{\kappa}^{n+1}$ and, by the volume-preserving
properties of the flow, such sphere has to enclose the same volume
as the initial condition $X_0(M)$.

Finally, from Proposition \ref{convergence of f} we can conclude
with the standard arguments as in \cite[Theorem 3.5]{Sch06},
\cite[Theorem 7.3]{CS10} and \cite[Corollary 7.3]{Mak12} that the
flow converges exponentially to the geodesic sphere $\mathscr{S}$ in
${\mathbb{H}}_{\kappa}^{n+1}$ in the $C^{\infty}$-topology.

\bibliographystyle{amsplain}

\end{document}